\documentclass{birkjour_t2}

\usepackage{amsfonts,amssymb}
\usepackage{graphicx}
\usepackage{cite}
\usepackage{enumerate}
\usepackage{multirow}
\usepackage{etex}
\usepackage[latin1]{inputenc}
\usepackage[T1]{fontenc}
\usepackage{amsfonts}
\usepackage{amssymb, amsmath}
\usepackage{geometry}
\usepackage{paralist}
\usepackage{lscape}
\usepackage[english]{babel} 
\usepackage{ifthen}
\usepackage{bbm}		
\usepackage{graphicx}
\usepackage{lmodern}
\usepackage{algorithmic}
\usepackage{float}
\usepackage{wrapfig} 
\usepackage{textcomp}
\usepackage{graphicx}
\usepackage{cancel}
\usepackage{mathrsfs}
\usepackage{upgreek}
\usepackage{url}
\usepackage{tikz}	
\usepackage{textpos}
\usepackage{mathtools} 
\usepackage{ragged2e}
\usepackage{tabularx}
\usepackage{nicefrac}
\usepackage{soul}
\usepackage{yhmath}

\def \h{\textup{\textbf{h}}}
\def \u{\textup{\textbf{u}}}

\def \T{\mathsf{T}}

\def \diag{\textup{diag}}

\def \veq{v^{\textup{eq}}}
\def \ueq{\textup{\textbf{u}}^{\textup{eq}}}
\def \feq{f^{\textup{eq}}}
\def \lambdaeq{\lambda^{\textup{eq}}}
\def \lambdaf{\lambda^{\textup{f}}}
\def \d{\textup{T},}
\def \q{a}
\def \entropy{\eta}
\def \entropyflux{\psi}

\def \R{\mathcal{R}}

\def \HAR{\textup{\textbf{H}}}
\def \d{\,\textup{d}}
\def \WAR{\textup{\textbf{W}}}

\def \wAR{\omega^{\textup{ARZ}}}

\def \v1{v^{(1)}}

\DeclareFontFamily{U}{mathx}{\hyphenchar\font45}
\DeclareFontShape{U}{mathx}{m}{n}{<-> mathx10}{}
\DeclareSymbolFont{mathx}{U}{mathx}{m}{n}
\DeclareMathAccent{\widebar}{0}{mathx}{"73}

\usepackage{epstopdf}
\usepackage{algorithmic}
\ifpdf
\DeclareGraphicsExtensions{.eps,.pdf,.png,.jpg}
\else
\DeclareGraphicsExtensions{.eps}
\fi

\theoremstyle{plain}
\newtheorem{theorem}{Theorem}[section]
\newtheorem{lemma}[theorem]{Lemma}
\newtheorem{corollary}[theorem]{Corollary}

\theoremstyle{remark}
\newtheorem{remark}[theorem]{Remark}

\theoremstyle{definition}
\newtheorem{example}[theorem]{Example}

\title[Multiscale and stability analysis]{Multiscale and stability analysis of an Aw-Rascle-Zhang-type model with driver-dependent acceleration}

\author[Stephan Gerster and Giuseppe Visconti]{Stephan Gerster\footnote{University of Stuttgart, (stephan.gerster@gmail.com).}
	and Giuseppe Visconti\footnote{Department of Mathematics ``Guido Castelnuovo'', Sapienza~University~of~Rome, (giuseppe.visconti@uniroma1.it).}}

\begin{document}

\maketitle

\begin{abstract}
We introduce a new  traffic flow model.  The main idea is to describe acceleration effects in the Aw-Rascle-Zhang model through an additional agent-dependent equation, thereby allowing individual driver characteristics to evolve dynamically rather than being prescribed solely by the traffic density. The resulting system provides a hyperbolic description of traffic dynamics and admits an explicit analysis of its characteristic structure, Riemann problem, and entropy properties. To investigate the stability of the proposed model, we perform a Chapman-Enskog expansion up to second order and derive the corresponding diffusive and higher-order corrections. This analysis reveals how the stability properties depend on the underlying velocity and pressure functions. Several examples are presented to illustrate the theoretical results.
\end{abstract}



	\section{Introduction} \label{sec:intro}
Mathematical models of traffic flow provide a framework for describing the collective dynamics of vehicles at different scales. Depending on the level of detail, traffic models are commonly formulated at the microscopic, mesoscopic~\cite{Dimarco2019,Dimarco2021}, or macroscopic scale~\cite{Coclite2005Traffic,Gong:2023,Piu:2022,Hayat:2023,Albeaik:2022,Marques:2013,Iannini:2016,Lu:2025}. Microscopic models describe the motion of individual vehicles and their interactions with the vehicles ahead, while macroscopic models describe traffic through averaged quantities such as density and velocity. A fundamental question is how the macroscopic evolution equations can be consistently derived from the underlying microscopic or mesoscopic vehicle dynamics~\cite{Borsche:2018,Burger:2018,Cristiani:2016,Cardaliaguet:2021,Dimarco2019,Dimarco2021,Emiliano2019,Cristiani2025}.

A classical example of a microscopic traffic model is the Follow-the-Leader  model~\cite{Gazis:1961}, in which the acceleration of a vehicle depends on the distance to the vehicle immediately ahead. Such models provide a direct description of individual vehicle interactions, but their discrete nature makes them less suitable for describing traffic evolution on macroscopic spatial and temporal scales. The derivation of macroscopic equations from microscopic car-following models therefore provides an important link between individual vehicle dynamics and continuum traffic flow.

The Lighthill-Whitham-Richards (LWR) model~\cite{Lighthill:1955,Richards:1956} is the classical first-order description among macroscopic traffic models. It assumes that the velocity is determined instantaneously by the local density through an equilibrium velocity-density relation. Although this assumption leads to a particularly simple conservation law, it does not explicitly describe the dynamics of velocity and requires the equilibrium relation to be prescribed in advance.

Second-order models extend this framework by treating velocity as an additional dynamic variable. In particular, the Aw-Rascle-Zhang (ARZ) model~\cite{Aw:2000,Zhang:2002} introduces a traffic hesitation function that modifies the relation between velocity and density. The ARZ model can be obtained as a macroscopic limit of suitable second-order microscopic models~\cite{Aw:2002}. This connection demonstrates that microscopic interaction rules can determine the structure of the corresponding macroscopic traffic equations.

In this work, we introduce a novel  traffic model, which provides a car-following description based on the interaction between consecutive vehicles. The presented \emph{PIU model} contains a velocity adaptation mechanism depending on the headway, and its structure allows the microscopic dynamics to be rewritten in terms of a suitable traffic pressure. This reformulation is particularly useful for identifying conserved quantities and for deriving a macroscopic model in the hydrodynamic limit.

The central objective is to establish the connection between the PIU microscopic dynamics and its corresponding macroscopic description. Starting from the ordinary differential equations governing the individual vehicles, we introduce the microscopic density and derive the evolution of the velocity and the associated traffic pressure. In particular, the microscopic dynamics lead 
to quantities that in Lagrangian coordinates  satisfy a transport equation, which provides the key structural relation for the macroscopic model.

Furthermore, we obtain a hyperbolic  hydrodynamic limit on the macroscopic level.  We transform the system to Eulerian coordinates and investigate its conservative and non-conservative formulations. The resulting system possesses a structure that allows for an explicit analysis of its characteristic speeds, eigenvectors, and Riemann invariants. Under appropriate assumptions on the pressure function, the system exhibits a particularly simple wave structure.

An important part of the analysis concerns the Riemann problem associated with the macroscopic PIU model. We characterize the elementary wave families in terms of a Temple class system and derive the corresponding shock and rarefaction curves. The resulting structure allows us to determine the intermediate states and to obtain an explicit description of the Riemann solution. This provides a direct connection between the microscopic PIU interaction law and the wave propagation properties of the macroscopic traffic model. 
We further investigate the entropy structure of the macroscopic system. In particular, we construct entropy-entropy flux pairs. 

The PIU model also provides a useful framework for studying the relation between microscopic velocity adaptation and macroscopic traffic stability. By considering suitable constitutive relations, we obtain explicit expressions for the characteristic speeds and the corresponding coefficients in the macroscopic equations.

To  investigate the stability properties of the macroscopic PIU model, we use a Chapman-Enskog expansion, i.e.~an asymptotic expansion around the local equilibrium manifold with respect to the small relaxation parameter, which allows  to derive reduced macroscopic equations together with their dissipative higher-order corrections, see e.g.~\cite{ChenLevermoreLiu1994,GoatinRizzo2024,Goatin2019}. We carry out this expansion up to second order. This higher-order expansion allows us to identify not only the leading-order diffusive correction, but also the higher-order terms governing the stability 
and propagation of perturbations around the equilibrium state. In particular, we derive explicit expressions for the corrections and analyze the resulting coefficients in terms of the underlying velocity and pressure functions.

Several examples of traffic relations are considered, including power-law and a  velocity-density relation that is specifically tailored  to  the memory-dependent nature of driver behavior and related hysteresis effects~\cite{Corli:2019,Corli:2023,Goatin2026}, particularly in congested or stop-and-go traffic regimes.

\medskip

\noindent
This paper is structured as follows: Section~\ref{SectionMicro} 
introduces the microscopic traffic flow model and the proposed PIU formulation. In Section~\ref{SecMacroDescription} the corresponding macroscopic model is derived and its characteristic structure, Riemann problem and entropy properties are analyzed. 
Section~\ref{SectionRelaxation} investigates relaxation and stability through a Chapman-Enskog expansion, including first- and second-order corrections. Section~\ref{SectionCalibration}
presents different equilibrium calibrations and illustrates the resulting stability conditions. Finally, Section~\ref{SectionNumerics} illustrates numerically the presented results.

\section{Microscopic description of traffic flow}\label{SectionMicro}
Most microscopic traffic flow models are based on a \textit{Follow-the-Leader}-ansatz, where the evolution of the position~$x_i(t)$ and the velocity~$v_i(t)$ of a car~$i$ at time~${t\geq 0}$ is described by ordinary differential equations~(ODEs) and depends only on the $(i+1)$-th vehicle ahead. In this section we briefly recall existing microscopic traffic flow models and introduce an extension of the microscopic description of the Aw-Rascle-Zhang model.

\subsection{Follow-the-leader models}\label{SecARZmicro}
A microscopic description of the car dynamics~\cite{Aw:2002} can be given by the following system of ODEs:
\begin{equation}\label{FtL}
	\begin{cases}
		\begin{aligned}
			\dot{x}_i&=v_i, \\
			\dot{v}_i&=c_\gamma \frac{\dot{x}_{i+1}-\dot{x}_{i}}{(x_{i+1}-x_i)^{\gamma+1}},
		\end{aligned}
	\end{cases}
\end{equation}
where~$\dot{x} = x'(t)$ denotes the time derivative. 
The local density around vehicle~$i$, which is a dimensionless quantity, and its inverse, the local normalized specific volume, are defined by
$$
\rho_i\coloneqq
\frac{\Delta X}{x_{i+1}-x_i}
\quad\text{and}\quad
\tau_i\coloneqq \frac{1}{\rho_i},
$$
where~$\Delta X$ is the length of a car.
The specific choices~$c_\gamma(\Delta X)\coloneqq (\Delta  X)^\gamma$ and $P(\rho)\coloneqq \nicefrac{\rho^\gamma}{\gamma}$ lead to the relation
\begin{equation*}	    
	-\dot{v}_i
	=
	-
	c_\gamma(\Delta X) \frac{\dot{x}_{i+1}-\dot{x}_{i}}{(x_{i+1}-x_i)^{\gamma+1}}
	=
	c_\gamma(\Delta X)
	\frac{1}{(x_{i+1}-x_i)^{\gamma-1}}
	\bigg(
	-
	\frac{\dot{x}_{i+1}-\dot{x}_i}{(x_{i+1}-x_i)^2}
	\bigg) 
	=
	P'(\rho_i)\dot{\rho}_i
	=
	\dot{P(\rho_i)},
\end{equation*}
where $\dot{P(\rho_i)} \coloneqq
\frac{\textup{d}}{\textup{d}t}
P\big( \rho_i(t) \big)$
denotes the composite time derivative. Hence, the relation~$\dot{\omega}_i=0$ holds for~$\omega_i\coloneqq v_i+P(\rho_i)$. 
Since the quantity~$\omega_i(t)=\omega_i(0)$ is constant over time, it is a driver-dependent quantity that corresponds to car~$i$. The macroscopic form reads as  
\begin{equation}\label{ARZ}
	\begin{cases}
		\begin{aligned}
			\partial_t \rho + \partial_x(\rho v)&=0,\\
			\partial_t v+ \Big(v-P'(\rho)\rho\Big)\partial_x v
			&=0.
		\end{aligned}
	\end{cases} 
\end{equation}

\subsection{Multiclass models}\label{SecMicroMulticlass}
The idea of multiclass models~\cite{Bagnerini:2003,Emiliano2021} is to associate to each driver  an additional~information~$h_i$,  which accounts for different driving characteristics. For instance, a car accelerates faster than a truck. 
Then, the Follow-the-Leader system~\eqref{FtL} is augmented with an additional equation and reads as
\begin{equation*}
	\begin{cases}
		\begin{aligned}
			\dot{x}_i&=v_i, \\
			\dot{\omega}_i&=0, \\
			\dot{h}_i &= 0
		\end{aligned}
	\end{cases}
	\text{with}
	\quad \
	\omega_i\coloneqq
	v_i+q(\rho_i,h_i).
\end{equation*}
Since the driver-dependent information~$h_i(t)=h_i(0)$ is constant over time, it can be seen as a known quantity that acts similarly to a parameter, which parametrizes the term~$q(\cdot,h_i)$. The \textbf{third-order model (TOM)}, which has been recently introduced in~\cite{Ostia26}, is based on the ansatz
\begin{equation}\label{GVSMicroscopic}
	\begin{cases}
		\begin{aligned}
			\dot{x}_i&=v_i, \\
			\dot{\omega}_i&=0, \\
			\dot{h_i} &= \dot{\h(\rho_i)}
		\end{aligned}
	\end{cases}
	\text{with}
	\qquad
	\omega_i\coloneqq
	v_i+q(h_i).
\end{equation}
By rewriting the second equation of the proposed model~\eqref{GVSMicroscopic} in terms of velocity, i.e.
$$
\begin{aligned}
\dot{v}_i&= -
\dot{q(h_i)}=
-q'(h_i)
\dot{h_i}
=
-q'(h_i)
\dot{\h(\rho_i)}
=
-q'(h_i)
\h'(\rho_i)
\dot{\rho_i}\\
&=
q'(h_i)
\h'(\rho_i)
\rho_i
\frac{v_{i+1}-v_i}{x_{i+1}-x_i},
\end{aligned}
$$
the macroscopic form reads
\begin{equation*}
	\begin{cases}
		\begin{aligned}
			\partial_t \rho + \partial_x(\rho v)&=0,\\
			\partial_t v+ \Big(v-q'(h)\h'(\rho)\rho\Big)\partial_x v
			&=0, \\
			\partial_t \big(h-\h(\rho)\big) + v \partial_x \big(h-\h(\rho)\big)&= 0. 
		\end{aligned}
	\end{cases} 
\end{equation*}

\subsection{Pressure-independent, univariate (PIU) closure} 
We extend the acceleration coefficient~$c_\gamma$ in the microscopic description of the Aw-Rascle-Zhang model~\eqref{FtL} to include  a driver-dependent property. It describes only driver-dependent information, for instance the fact that a car accelerates faster than a truck. It is independent of the traffic pressure and local density of cars. 
The univariate~function~$a(h_i)$ requires a univariate closure, which is included as~$\dot{h}_i=0$. The complete system reads as
\begin{equation}\label{AnsatzPiu}
	\begin{cases}
		\begin{aligned}
			\dot{x}_i&=v_i, \\
			\dot{v}_i&= c_\gamma(\Delta X)
			\q(h_i) \Big(
			\frac{\rho_i}{\Delta X}
			\Big)^\gamma  \frac{v_{i+1}-v_{i}}{x_{i+1}-x_i},\\
			\dot{h}_i &= 0.
		\end{aligned}
	\end{cases}
\end{equation}
The microscopic description~\eqref{AnsatzPiu} allows for the conservative form
\begin{equation}\label{PIUconservative}
\left\{
		\begin{aligned}
			\dot{x}_i&=v_i, \\
			\dot{\omega}_i&=0, \\
			\dot{h}_i &= 0
		\end{aligned}
	\right.
	\text{with} \ \
	\omega_i\coloneqq
	v_i+\q\big(h_i(0)\big)P(\rho_i), \
	P(\rho_i)\coloneqq \begin{cases}
		\ln(\rho_i) & \text{if } \gamma =0, \\
		\frac{\rho_i^\gamma}{\gamma}  & \text{if } \gamma>0. 
	\end{cases}
\end{equation}

	\section{Macroscopic description of traffic flow}\label{SecMacroDescription}
The macroscopic description of the microscopic Follow-the-Leader and multiclass models in Section~\ref{SecARZmicro} and~\ref{SecMicroMulticlass} has been ensured by a hydrodynamic limit~\cite{Aw:2002,Bagnerini:2003}. The same arguments yield the corresponding formulation of the proposed microscopic model \eqref{PIUconservative} on a macroscopic level, when the number of cars tends to infinity and their size~$\Delta X\rightarrow 0$ shrinks to zero. 
We emphasize that the  proofs~\cite{Aw:2002,Bagnerini:2003} are based on the fact that the systems in Lagrangian coordinates are conservative. 
This issue would be non-trivial if~the term $h_i$ changed over time, which is the case in Eulerian coordinates. In Lagrangian coordinates, however, the driver-dependent quantities~$h_i(t)=h_i(0)$ are constant. 
In particular, we directly obtain from the system of ordinary differential equations~$\dot{\omega}_i=0$ and~$\dot{h}_i=0$ the partial differential equations~$\partial_T \big[v + a(h)P(\rho)\big]=0$ and $\partial_T h =0$.  
Here, the variables $(T,X)$ denote the Lagrangian time and space coordinates. Those are related to the Eulerian~$(t,x)$ by~$X(t,x) = \int^x \rho(t,\xi) \d \xi$. 
By including the conservation of mass, the hydrodynamic limit~$i\Delta X\to X$  reads in Lagrangian coordinates as
\begin{equation}\label{PiuLagrange}
	\begin{cases}
		\begin{aligned}
			\partial_T \tau(T,X) - \partial_X v(T,X) \ \ &=0,\\
			\partial_T v(T,X) -
			\frac{\q(h)}{\tau^{\gamma}} \big(T,X\big)
			\frac{\partial_X v}{\tau} (T,X)&=0,\\
			\partial_T h(T,X) &=0
		\end{aligned}
	\end{cases}
	\text{for}\quad
	\tau\coloneqq
	\frac{1}{\rho}.
\end{equation}	
Note that this Lagrangian form is directly seen from the microscopic description~\eqref{AnsatzPiu} by considering the limit
\begin{equation*}	
	\frac{v_{i+1}-v_{i}}{x_{i+1}-x_i}
	\rightarrow 
	\partial_x v=
	\frac{\partial_X v}{\tau} 
	\quad\text{for}\quad \Delta X \rightarrow 0.
\end{equation*} 
However, this is in general not proven 
and we make no use of the  convergence with respect to  non-conservative forms. 
	Following the lines of~\cite{Aw:2002,Bagnerini:2003},  the system~\eqref{PiuLagrange} in Eulerian coordinates reads as
\begin{equation}\tag{\textup{P1}}\label{PiuS1}
	\left\{
	\begin{aligned}
		\partial_t \rho + \partial_x(\rho v) &= 0, \\
		\partial_t v + \Big(v-a(h)P'(\rho)\rho \Big)\partial_x v &= 0, \\
		\partial_t h + v \partial_xh &= 0.
	\end{aligned}
	\right.
\end{equation}

\noindent
Hence, the system~\eqref{PiuS1} can be written in the quasilinear form 
$$
\partial_t \u + Q(\u) \partial_x \u=0 
\quad\text{with}\quad 
Q(\u)
=
\begin{pmatrix}
	v & \rho & 0 \\
	0 & v-a(h)P'(\rho)\rho & 0     \\
	0 & 0 & v
\end{pmatrix}
$$
and~$\u = (\rho,v,h)^\T$. 
The eigenvalue decomposition
$L(\u) Q(\u)  R(\u)= D(\u) 
$ 
reads 
\begin{align*}
	R(\u)&=
	\begin{pmatrix}
		1 & 0 & \rho \\
		0 & 0 & -a(h)P'(\rho)\rho   \\
		0 & 1 & 0
	\end{pmatrix}, \\
	L(\u)&=
	\big[a(h)P'(\rho)\rho\big]^{-1}
	\begin{pmatrix}
		a(h)P'(\rho)\rho & \rho & 0 \\
		0 & 0              & a(h)P'(\rho)\rho \\
		0 & -1 & 0
	\end{pmatrix}
\end{align*}
with  eigenvalues 
$D(\u) = \diag\big\{ v,v,v-a(h)P'(\rho)\rho \big\}$ 
that are labeled as 
$$\lambda_1(\rho,v,h)=v-a(h)P'(\rho)\rho
\quad\text{and}\quad \lambda_2(\rho,v,h) = \lambda_3(\rho,v,h) = \lambda_v(v) \coloneqq v.$$
Define the quantities
$$
z=\rho\omega, \ \
w=\rho h \ \
\text{ with }
\ \
\omega = v+a(h)P(\rho) 
\ \
\text{ and }
\ \
v(\rho,z,w)
=
\frac{z}{\rho}-a\Big(\frac{w}{\rho}\Big) P(\rho). 
$$
It has also the following conservative form, which is proven in the Lemma~\ref{LemmaConservative}.
\begin{equation}\label{PiuS3}\tag{\textup{P2}}
\left\{
		\begin{aligned}
			\partial_t \rho  + \partial_x\Big(\rho v(\rho,z,w) \Big)&=0,\\
			\partial_t z
			+
			\partial_x \Big(
			z v(\rho,z,w) 
			\Big)
			&=0, \\
			\partial_t w +
			\partial_x \Big(
			w v(\rho,z,w) 		 
			\Big)&= 
			0	\end{aligned}
\right.
\,	\text{ or } \
\left\{
		\begin{aligned}
			\partial_t \rho + \partial_x\Big(\rho v(\rho,z,w)\Big)&=0,\\
			\partial_tz
			+
			\partial_x \Big(
			\frac{z^2}{\rho}
			-z
			a(\nicefrac{w}{\rho}) P(\rho)
			\Big)
			&=0, \\
			\partial_t w +
			\partial_x \Big(wv(\rho,z,w) 
			\Big)&= 
			0	\end{aligned}
\right.
\end{equation}

\begin{lemma}[Conservative form]\label{LemmaConservative}
	The system~\eqref{PiuS1} is for smooth solutions equivalent to the conservative formulation~\eqref{PiuS3}. Furthermore, it holds
	$$
	\partial_t \omega + v\partial_x \omega = 0 
	\quad\text{for}\quad
	\omega = v + a(h) P(\rho).
	$$	
\end{lemma}

\begin{proof}
	The  equation~$\partial_t v+v\partial_xv
	=a(h)P'(\rho)\rho\partial_x v$ in  model~\eqref{PiuS1} yields
	\begin{equation*}
		\partial_t\omega+v\partial_x\omega
		= \partial_t v+v\partial_x v
		+a'(h)P(\rho)\big(\partial_t h+v\partial_x h\big)
		-a(h)P'(\rho)\rho\partial_x v=0.
	\end{equation*}
	Combining this equation with the mass conservation law in the system~\eqref{PiuS1}, we obtain
	\begin{equation*}
		\partial_t(\rho\omega )+\partial_x(\rho\omega v)
		=
		\omega\big(\partial_t\rho+\partial_x(\rho v)\big)
		+\rho\big(\partial_t\omega+v\partial_x\omega\big)
		=0.
	\end{equation*}

	\noindent
	Conversely, suppose that the system~\eqref{PiuS3} holds. We have
	\begin{equation*}
		0
		=\partial_t(\rho\zeta)+\partial_x(\rho\zeta v)\\
		=\zeta\big(\partial_t\rho+\partial_x(\rho v)\big)
		+\rho\big(\partial_t\zeta+v\partial_x\zeta\big)
		\quad\text{for}\quad
		\zeta \in \{\omega, h\}. 
	\end{equation*}
	Hence, the equalities~$
	\partial_t\omega+v\partial_x\omega=0$ and $
	\partial_t h+v\partial_xh=0$ hold, which  yield the remaining second equation 
	$
	\partial_t v+
	\big(v-a(h)P'(\rho)\rho\big)\partial_xv=0.
	$
\end{proof}

\subsection{Discussion of the Riemann problem}
Following~\cite[Sec.~5.1]{BRESSAN}, 
Theorem~\ref{TheoremRP} 
presents the entropy admissible weak solution
to a Riemann problem with constant left~$\u_\ell$
and right~$\u_r$
state.  
The PIU system~\eqref{PiuS1},~\eqref{PiuS3} forms a Temple class system where the fields $k=2,3$, which correspond to the double eigenvalue~$\lambda_v(v)=v$  and which travel faster than the first wave,
form contacts. 
Since the first field is genuinely nonlinear, it can  either produce a rarefaction or a shock wave, which is connected by an intermediate state~$
(
\bar{\rho},\bar{v},\bar{h}
)^\T$ 
with the contact.

\medskip

\begin{theorem}[Solution to the Riemann problem]\label{TheoremRP}
	Consider the conservative system
	\begin{equation*}
		\partial_t \u+\partial_x f(\u)=0
		\ \
		\text{with}
		\ \
		\u=
		\begin{pmatrix}
			\rho \\ z \\ w
		\end{pmatrix}
		\ \
		\text{for}
		\ \
		z=\rho\omega,
		\ \
		w=\rho h,
		\ \
		\omega=v+a(h)P(\rho)
	\end{equation*}
	with initial data~$\u(0,x)=\u_\ell$ for $x<0$ and $\u(0,x)=\u_r$ for $x>0$. 
	Define the critical right-state velocity by
	$$
	v_{\textup{crit}}
	\coloneqq
	v_\ell
	+
	a(h_\ell)
	\Big(P(\rho_\ell)-P(\rho_r)\Big).
	$$
	The solution to the Riemann problem with 
	initial velocity~$v_r > v_{\textup{crit}}$ 
	reads as 
	$$
	\u(t,x)
	\coloneqq
	\begin{cases}
		\u_\ell
		& \text{for \ \ } x<t \lambda_1(\u_\ell), \\
		\u_{\textup{rf}}(t,x)
		& \text{for \ \ } x\in \big( t\lambda_1(\u_\ell),t\lambda_1(\bar{\u})  \big) 
		\text{ \ \ and \ \  } \nicefrac{x}{t}   =\lambda_1\big(
		\u_{\textup{rf}}(t,x)
		\big), \\
		\bar{\u}
		& \text{for \ \ } x\in \big( t\lambda_1(\bar{\u}),tv_r  \big), \\
		\u_r
		& \text{for \ \ } x>t v_r.
	\end{cases}
	$$

	\noindent
	The self-similar solution~$\u_{\textup{rf}}
	= \big(
	\rho_{\textup{rf}}, 
	z_{\textup{rf}},
	w_{\textup{rf}}
	\big)^\T
	$ in the rarefaction span reads as
	$$
	\rho_{\textup{rf}}(x,t)=
	\begin{cases}
		\displaystyle
		\exp\!\left(
		\frac{\omega_\ell-\nicefrac{x}{t}}{a(h_\ell)}-1
		\right)
		&\text{if \ } \gamma=0,\\
		\left[
		\frac{\gamma\left(\omega_\ell-\nicefrac{x}{t}\right)}
		{(1+\gamma)a(h_\ell)}
		\right]^{\nicefrac{1}{\gamma}}
		&\text{if \ } \gamma>0,
	\end{cases}
	$$
	$z_{\textup{rf}}(t,x)
	=
	\rho_{\textup{rf}}(t,x) \omega_\ell
	$
	and~$w_{\textup{rf}}(t,x)
	=
	\rho_{\textup{rf}}(t,x) h_\ell
	$.  
	The intermediate state~$\widebar{\u}$ is given by~
	$$
	\widebar{\rho}
	=
	\begin{cases}
		\rho_\ell
		\exp\!\left(
		\dfrac{v_\ell-v_r}{a(h_\ell)}
		\right)
		& \text{if \ } \gamma=0,\\
		\left(
		\rho_\ell^\gamma+
		\dfrac{\gamma(v_\ell-v_r)}{a(h_\ell)}
		\right)^{\nicefrac{1}{\gamma}}
		& \text{if \ } \gamma>0
	\end{cases}
	\qquad
	\text{and}
	\qquad
	\widebar{v}=v_r, 
	\quad
	\widebar{h}=h_\ell.
	$$
	
	\noindent
	In the other case
	$v_r < v_{\textup{crit}}$, 
	the solution contains a Lax shock connecting the left state~$\u_\ell$  to the intermediate state~$\widebar{\u}$ where  the propagation speed reads as
	$$
	s = \frac{v_\ell - \alpha v_r}{1-\alpha}
	\quad
	\text{with}
	\quad
	\alpha = 
	\begin{cases}
		\exp\!\left(\frac{v_\ell-v_r}{a(h_\ell)}\right)
		&
		\text{for }\ \
		\gamma=0,
		\\
		\left(
		1+\frac{\gamma(v_\ell-v_r)}
		{a(h_\ell)\rho_\ell^\gamma}
		\right)^{\nicefrac{1}{\gamma}}
		&
		\text{for }\ \
		\gamma>0.
	\end{cases}
	$$
	
	\noindent
	In particular, the system is of Temple class. The Hugoniot loci coincide with the corresponding rarefaction curves. Both shock and rarefaction waves connect the left state to the same intermediate state. 
	It is 
	followed by a contact discontinuity connecting the intermediate state to the right state.
\end{theorem}

\medskip

\begin{proof}
	Due to the Riemann invariant of the first and third characteristic field, i.e.~${
		\R_1(\u)=\omega}$, 
	$\R_3(\u)=v$, 
	the intermediate state satisfies~$\widebar{\omega} = \omega_\ell$ and $\widebar{h}=h_\ell$. 
	Since it is connected to the right state by a contact, we have~$\widebar{v}=v_r$ such that the density is determined by~$P(\widebar\rho) = (\omega_\ell-v_r)a(h_\ell)^{-1}$.  
	The velocity remains constant over a contact wave and hence the velocity of the intermediate state satisfies~$\bar{v}=v_r$.	
	The rarefaction wave is parameterized by~$\xi=\nicefrac{x}{t}$. The self-similar solution is given by
	\begin{align*}
		&	\xi=\lambda_1\big(\rho(\xi)\big)
		= 
		\omega_\ell - a(h_\ell) \Big[P'\big(\rho(\xi)\big)\rho(\xi)+P\big(\rho(\xi)\big)\Big]\\
		&=
		\begin{cases}
			\omega_\ell - a(h_\ell) \Big[1+\ln\big(\rho(\xi)\big) \Big] 
			& \text{if } \gamma =0, \\
			\omega_\ell - a(h_\ell) \frac{1+\gamma}{\gamma} \rho(\xi)^\gamma 
			& \text{if } \gamma >0 \\
		\end{cases} \\
		& \Leftrightarrow \quad
		P\big(\rho(\xi)\big) = 
\begin{cases}
	\frac{\omega_\ell -\xi}{a(h_\ell)}-1
	& 		\text{for \ \ }  \gamma =0, \\
	\frac{\omega_\ell - \xi}{(1+\gamma)a(h_\ell)}
		& 		\text{for \ \ }  \gamma >0. \\
\end{cases}		
	\end{align*}
The velocity in the rarefaction wave is obtained from~$v(\xi) = \omega_\ell - a(h_\ell)P\big(\rho(\xi)\big)$.

	\medskip
	\noindent	
	When the first-family wave is replaced by the
	corresponding Lax shock, whose speed is determined by the
	Rankine-Hugoniot condition 
	\begin{equation}\label{ProofHL}
		s(\u_\ell - \widebar{\u}) = f(\u_\ell) - f(\widebar{\u})
		= v_\ell \u_\ell- v_r \widebar{\u} 
		\quad\Leftrightarrow\quad
		\widebar{\u}= \alpha\u_\ell, \ \
		\alpha = \frac{v_\ell-s}{v_r-s},
	\end{equation}
	where we use  $\widebar{v}=v_r$, since the second and third waves are  contact
	discontinuities propagating at speed $v_r$.  In particular, it holds
	\[
	\alpha
	=
	\frac{1}{\rho_\ell}
	P^{-1}\!\left(
	P(\rho_\ell)
	+
	\frac{v_\ell-v_r}{a(h_\ell)}
	\right).
	\]
	The linear relation~\eqref{ProofHL}, which results from the Rankine-Hugoniot condition, states the intermediate state
	$$
	\widebar{\omega}
	=
	\frac{z}{\rho}
	=
	\frac{\alpha z_\ell}{\alpha \rho_\ell}
	=
	\omega_\ell.
	$$	
	Hence, the Hugoniot locus and the corresponding integral curve coincide. 
	Along the first integral curve, the invariants $\omega$ and $h$ remain constant. The characteristic speed
	$\lambda_1(\rho)
	=\omega_\ell-a(h_\ell)\bigl(P(\rho)+\rho P'(\rho)\big)
	$ 
	is strictly decreasing in the density~$\rho$, 
	since it holds~$
	v=\omega_\ell-a(h_\ell)P(\rho)
	$. 
	 Therefore, the Lax entropy condition
	$
	\lambda_1(\widebar{\rho}_s)<s<\lambda_1(\rho_\ell)
	$, which yields a shock, 
	is equivalent to
	$
	{\widebar{\rho}_s>\rho_\ell}
	$. Here, the state~$\widebar{\rho}_s$ denotes the intermediate density that is obtained when the left state is connected by a shock wave. 
	The condition for a $1$-rarefaction wave can be written as
	$$
	v_r-v_\ell
	>
	a(h_\ell)\Big(P(\rho_\ell)-P(\rho_r)\Big)
	\quad
	\Leftrightarrow
	\quad
	v_r>v_{\textup{crit}}. 
	$$
	Hence, the  condition for the $1$-rarefaction wave is proved.
\end{proof}

\medskip

\begin{remark}[Relation to the ARZ and LWR model]\label{RemarkARZandLWR}
	The wave classification is closely related to the classical Aw-Rascle-Zhang 
	model~\eqref{ARZ}. There, the first Riemann invariant is 
	$
	\wAR =v+p(\rho). 
	$
	Since the Riemann invariant $\wAR$ is constant along a $1$-wave curve, it holds~$
	v_r-v_\ell=p(\rho_\ell)-p(\rho_r).
	$
	Consequently, if in the Aw-Rascle-Zhang model the condition
	\begin{equation}\label{CondARZ}
		2p'(\rho)+\rho p''(\rho)>0
	\end{equation}
	is satisfied, the eigenvalue  $
	\lambda_1^{\textup{ARZ}}(\rho,v) =v-\rho p'(\rho)	
	=
	\omega^{\textup{ARZ}}- p(\rho) - \rho p'(\rho)	$
	is strictly decreasing with respect to $\rho$ along the
	$1$-wave curve. Hence, we have the implications
	\begin{equation}\label{DensityOrdering}
		\rho_\ell>\rho_r
		\ \ \Leftrightarrow\ \
		\text{$1$-rarefaction wave}
		\ \quad\quad\text{and}\quad\quad \
		\rho_\ell<\rho_r
		\ \ \Leftrightarrow\ \
		\text{$1$-shock}.
	\end{equation}
	The PIU model is related to this property directly, since along a $1$-wave curve with $h=h_\ell$ 
	constant, the corresponding traffic pressure function is
	\[
	p(\rho)=a(h_\ell)P(\rho).
	\]
	Since we have the property $a(h_\ell)>0$, the corresponding genuine nonlinearity condition
	reduces to the condition
	$
	2P'(\rho)+\rho P''(\rho)>0
	$, which is similar to condition~\eqref{CondARZ} of the Aw-Rascle-Zhang model. In fact, this condition is satisfied for both choices~$
	P(\rho)=\ln(\rho)$ 
	and  
	$
	P(\rho)=\nicefrac{\rho^\gamma}{\gamma}$ for $\gamma>0$.
	Therefore, for states lying on the same $1$-Riemann invariant curve, the
	wave type can equivalently be characterized by the same implications, namely the case~$\rho_\ell>\rho_r$ yields a $1$-rarefaction, while the case $\rho_\ell<\rho_r$ yields a $1$-shock. 
	
	For general Riemann data with independently prescribed velocities $v_\ell$ and
	$v_r$, however, the density ordering alone is not sufficient.  In that
	case the classification must be expressed in terms of the critical
	velocity in Theorem~\ref{TheoremRP}.

	This issue motivates to introduce an equilibrium relation that is defined in terms of a closure consistent
	with a prescribed LWR fundamental diagram. Let $\HAR(\rho)$ denote an
	equilibrium closure function for the driver-dependent variable $h$, i.e.~$
	h=\HAR(\rho). 
	$
	Fixing an equilibrium value $\omega_0$ of the Riemann invariant 
	$
	\omega=v+a(h)P(\rho),
	$ 
	we require the corresponding equilibrium velocity to coincide with a
	prescribed LWR velocity function~$V(\rho)$, namely, 
	\begin{equation}
		\begin{aligned}
			\label{Hchoice}
			&V(\rho)
			=
			\omega_0-a\big(\HAR(\rho)\big)P(\rho)
			\quad
			\Leftrightarrow
			\quad
			\HAR(\rho)
			=
			a^{-1}\!\left(
			\frac{\omega_0-V(\rho)}{P(\rho)}
			\right) \\
			&\text{with equilibrium}\quad
			Q(\rho)
			\coloneqq
			a\big(\HAR(\rho)\big)P(\rho)
			=\omega_0-V(\rho),
		\end{aligned}
	\end{equation}
	provided that the function $a()$ is locally invertible. 
	With this choice, the equilibrium PIU velocity is exactly the prescribed LWR velocity~$V(\rho)$. 
	This closure  establishes a direct correspondence between the
	characteristic speeds of the PIU equilibrium system and the LWR model.
	Indeed, the choice~\eqref{Hchoice} 
	yields  the relations
	$
	Q'(\rho)=-V'(\rho)
	$ 
	and 
	$
	\lambda^{\textup{LWR}}(\rho)
	=
	V(\rho)-\rho Q'(\rho)
	=
	V(\rho)+\rho V'(\rho),
	$
	which is precisely the characteristic speed of the LWR model
	$
	\partial_t\rho+\partial_x\big(\rho V(\rho)\big)=0.
	$ 
	Consequently, in this equilibrium, the wave type can be
	characterized even  for general Riemann data by the desired density ordering~\eqref{DensityOrdering}. 
\end{remark}

		\subsection{Entropy-entropy flux pairs}
Theorem~\ref{LemmaEntropy} states entropy-entropy flux pairs for the PIU model~\eqref{PiuS3}. 

\medskip
\begin{theorem}[Semi-convex entropy-entropy flux pairs]\label{LemmaEntropy}	
	Let $F\in C^2(\mathbb{R}^2)$ be any strictly convex function. 
	Then, the functions 
	$$
	\entropy(\rho,z,w)
	=
	\rho F\left(\frac{z}{\rho},\frac{w}{\rho}\right)
	\quad\text{and}\quad
	\entropyflux(\rho,z,w)
	=v(\rho,z,w)
	\entropy(\rho,z,w)
	$$
	form a semi-convex entropy-entropy flux pair that satisfies the entropy equation exactly, i.e.~$
	\partial_t \entropy(\rho,z,w) + \partial_x \entropyflux(\rho,z,w)=0. 
	$ In particular, the choices
	$$
	\begin{aligned}
		&F(\omega,h)=\widetilde{F}(\omega), \quad  F(\omega,h)=\widetilde{F}(h),\quad 
		F(\omega,h)=\frac{1}{2}\omega^2+E(h), \\
		&\text{i.e.}\quad
		\entropy(\rho,z,w) = \frac{z^2}{2\rho} + \rho E\left( \frac{w}{\rho} \right),
	\end{aligned}
	$$
	yield semi-convex entropies for all strictly convex functions~$\widetilde{F}, E \in C^2(\mathbb{R})$.

\end{theorem}

\medskip
\begin{proof}
	For any convex function $F\in C^2(\mathbb{R}^2)$ we have
	$$
	\partial_t F(\omega,h)
	+v\partial_x F(\omega,h)
	=
	F_\omega(\omega,h)
	\left(\partial_t\omega+v\partial_x\omega\right)\\
	+
	F_h(\omega,h)
	\left(\partial_t h+v\partial_x h\right)
	=0.
	$$
	This relation and conservation of mass imply the entropy equality
\begin{align*}
	0&=
	F(\omega,h)
\big(
\partial_t\rho+\partial_x(\rho v)
\big)
+
\rho
\big(
\partial_tF(\omega,h)
+v\partial_xF(\omega,h)
\big) \\
&=
	\partial_t\big(\rho F(\omega,h)\big)
	+\partial_x\big(\rho vF(\omega,h)\big).
\end{align*}
	The entropies are semi-convex. Namely, for all~$\xi = (\xi_\rho, \xi_z, \xi_w)^\T $ it holds
	$$
	\xi^\T \nabla^2\eta\,\xi
	=
	\frac{1}{\rho}
	\begin{pmatrix}
		\xi_z-\omega\xi_\rho\\
		\xi_w-h\xi_\rho
	\end{pmatrix}^{\T}
	\nabla^2F(\omega,h)
	\begin{pmatrix}
		\xi_z-\omega\xi_\rho\\
		\xi_w-h\xi_\rho
	\end{pmatrix} \geq 0.
	$$
\end{proof}

\medskip

The origin of these entropy functions can be traced back to the
microscopic formulation, where the quantities
$
\dot{\omega}_i=0 
$ and 
$
\dot{h}_i=0
$ are preserved, when they are transported unchanged along the
individual vehicle trajectories. Consequently, for any smooth function
$F$, the quantities $F(\omega_i)$ and $F(h_i)$ are likewise preserved
along the microscopic trajectories.
Since $\rho$ represents the number of vehicles per unit length, the terms
$\rho F(\omega)$ and $\rho F(h)$ can be interpreted as densities of the
corresponding transported quantities. 
The fact that these entropies are only semi-convex in the full
conservative state is  a  consequence of the
structural form that is inherited from the microscopic invariants.

The kinetic entropy ansatz~$F(\omega,h)=\nicefrac{\omega^2}{2}+E(h)$  can be motivated directly from the
microscopic description and its physical interpretation. At the
microscopic level, each vehicle is characterized by its position,
its generalized velocity~$\omega=v+a(h)P(\rho)$ and the driver-dependent property~$h$. 
Since $\omega$ is the dynamically relevant velocity variable, we associate
with it the  quadratic kinetic energy. The specific kinetic energy
per vehicle is therefore~$
e_{\textup{kin}}(\omega)=\nicefrac{\omega^2}{2}
$. 
Moreover, the additional microscopic variable $h$ represents an internal
state of the vehicle or driver and can contribute an additional
specific energy. We denote this contribution by $E(h)$ and obtain the 
specific energy associated with one vehicle by
$
F(\omega,h)
=
e_{\textup{kin}}(\omega)+E(h)$.  
To pass from the microscopic description to the macroscopic one, the
specific energy has to be multiplied by the vehicle density~$\rho$.
More precisely, 
the product $\rho e$ represents the energy per unit length.

		\section{Relaxation and stability analysis}
\label{SectionRelaxation} 
Motivated by the  equilibrium relation $h=\HAR(\rho)$, whose LWR-consistency  is established in Remark~\ref{RemarkARZandLWR}, we introduce a relaxation mechanism that drives the system towards an equilibrium and analyze the stability of the resulting macroscopic relaxation system. 
We investigate an equilibrium that can be described by a second-order model, which is consistent with the first-order LWR model with the equilibrium velocity presented in equation~\eqref{Hchoice}. More precisely, we consider from now on the equilibria
$$
\veq(\rho,z)
=
\frac{z}{\rho}-Q(\rho)
=
\frac{z}{\rho}-a\big(\HAR(\rho)\big)P(\rho)
=
v\big(\rho,z,\HAR(\rho)\big),
$$
which are based on the traffic pressure~$P(\rho)=\nicefrac{\rho^\gamma}{\gamma}$ for $\gamma>0$. 
The  mathematical tool to study this second-order equilibrium is a Chapman-Enskog expansion,   introduced in the following Lemma.

\medskip

\begin{lemma}[First-order Chapman-Enskog expansion]\label{Oldpaper}
	If the PIU closure $a\in C^2$ is twice continuously differentiable, equations
	$$	
	\partial_t w +  \partial_x \Big(wv(\rho,z,w)
	\Big)
	=
	-
	\frac{w-\WAR(\rho)}{\varepsilon}	
	\quad\text{and}\quad
	\partial_t h +  v \partial_x h
	=
	-
	\frac{h-\HAR(\rho)}{\varepsilon}.
	$$
	imply for~$\HAR'\in C^1(\mathbb{R}^+)$ 
	and~$\WAR(\rho) = \rho\HAR(\rho)$ 
	the second-order representations
	\begin{alignat*}{8}
		& &&h &&=
		&&\HAR(\rho) &&+ \varepsilon 
		&&\HAR'(\rho)\rho
		\partial_x v\big(\rho,z,\WAR(\rho)\big)
		&&+
		\mathcal{O}\big(\varepsilon^2\big), \\
		&	a(&&h)&&=
		a\big(&&\HAR(\rho)\big)
		&&+
		\varepsilon 
		a'\big(\HAR(\rho)\big)
		&&	\HAR'(\rho)\rho
		\partial_x v\big(\rho,z,\WAR(\rho)\big)		&&+		\mathcal{O}\big(\varepsilon^2\big). 
	\end{alignat*}

\end{lemma}	

\begin{proof}
	We apply the Chapman-Enskog expansion~$w = \WAR(\rho) + \varepsilon w^{(1)} + \mathcal{O}\big(\varepsilon^2\big)
	$ to the third equation of system~\eqref{SystemRelaxAwRascle} which yields also an expansion in terms of hesitation
	\begin{equation}\label{hChapman}
		h=\frac{w}{\rho}
		=
		\frac{\WAR(\rho)}{\rho}
		+
		\mathcal{O}(\varepsilon)
		=
		\HAR(\rho)
		+
		\mathcal{O}(\varepsilon)
		\ \
		\text{with}
		\ \
		\HAR(\rho)
		=
		\frac{\WAR(\rho)}{\rho}.
	\end{equation}
	Taylor's theorem states for some~$\xi\in\big(h,\HAR(\rho) \big)$ or~$\xi\in\big(\HAR(\rho), h\big)$, respectively, which implies the bound~$ (h-\xi)=\mathcal{O}(\varepsilon)
	$, 
	the second-order representation
	\begin{equation}\label{Taylor}
		\begin{aligned}
			a(h)
			&=
			a\big(\HAR(\rho)\big)
			+
			\big(
			h-\HAR(\rho)
			\big)
			\Big[
			a'\big(
			\HAR(\rho)
			\big)
			+
			a''(\xi) (h-\xi)
			\Big] \\
			&=
			a\big(\HAR(\rho)\big)
			+
			\big(
			h-\HAR(\rho)
			\big)
			a'\big(\HAR(\rho)\big)
			+
			\mathcal{O}\big(\varepsilon^2\big).
		\end{aligned}
	\end{equation}
	The representation~\eqref{Taylor} and the expansion of~\eqref{hChapman} imply for the velocity the first-order representation
	$$
	v(\rho,z,w)
	=
	\frac{z}{\rho}-a(h) P(\rho)	
	=
	\frac{z}{\rho}
	-
	a\big(
	\HAR(\rho)
	\big)P(\rho)
	+
	\mathcal{O}(\varepsilon)
	=
	v\big(
	\rho,z,\WAR(\rho)
	\big)
	+
	\mathcal{O}(\varepsilon).
	$$

	\noindent
	Hence, the third equation satisfies
	\begin{alignat*}{8}
		&-w^{(1)}
		&&=\partial_t \WAR(\rho) +  \partial_x \Big(\WAR(\rho) v\big(\rho,z,\WAR(\rho)\big)
		\Big)
		&&+\mathcal{O}(\varepsilon) \\
		& &&=
		-\WAR'(\rho)\partial_x\Big(\rho v\big(\rho,z,\WAR(\rho)\big) \Big) \\
		&	&&\quad +
		\WAR'(\rho)  v\big(\rho,z,\WAR(\rho)\big)\partial_x \rho + \WAR(\rho) \partial_x v\big(\rho,z,\WAR(\rho)\big)
		&&+
		\mathcal{O}(\varepsilon) \\
		& &&=-\WAR'(\rho)\rho \partial_x v\big(\rho,z,\WAR(\rho)\big)
		\; +
		\WAR(\rho) \partial_x v\big(\rho,z,\WAR(\rho)\big)
		&&+
		\mathcal{O}(\varepsilon)
	\end{alignat*}
	which leads to  the second-order expansions
	\begin{alignat*}{8}
		&	w&&=
		\	\WAR(\rho)
		&&
		&&	+
		\varepsilon
		\big(
		\WAR'(\rho)\rho-\WAR(\rho)
		\big)
		&&\partial_x v\big(\rho,z,\WAR(\rho)\big)
		&&+
		\mathcal{O}\big(\varepsilon^2\big), \\
		h=	& \frac{w}{\rho} &&    		=
		\;	\frac{\WAR(\rho)}{\rho}
		&&
		&&	+
		\varepsilon\bigg(
		\WAR'(\rho)-\frac{\WAR(\rho)}{\rho}
		\bigg)
		&&\partial_x v\big(\rho,z,\WAR(\rho)\big)
		&&+\mathcal{O}\big(\varepsilon^2\big)\\
		& &&= \
		\HAR(\rho) && &&+ \varepsilon \rho\quad
		\HAR'(\rho)
		&&\partial_x v\big(\rho,z,\WAR(\rho)\big)
		&&+
		\mathcal{O}\big(\varepsilon^2\big).
	\end{alignat*}
	Then, the second-order Taylor  representation~\eqref{Taylor} reads as
	\begin{equation*}
		a(h)=
		a\big(\HAR(\rho)\big)
		+
		\varepsilon \rho
		a'\big(\HAR(\rho)\big)
		\HAR'(\rho)
		\partial_x v\big(\rho,z,\WAR(\rho)\big)		+		\mathcal{O}\big(\varepsilon^2\big).
	\end{equation*}
\end{proof}

\subsection{Relaxation in terms of the conservative form}
This section introduces the relaxation approach  to analyze the stability of the proposed traffic model. 
Regarding results on how 
solutions to  viscous approximations  converge to a weak
solution to  hyperbolic system we refer to~\cite{Coclite2021a,Coclite2010VanishingVF}

\medskip
\begin{theorem}[Diffusive relaxation limit]\label{TheoremAwRascle1}	
	Let an equilibrium~$
	\WAR(\rho)
	\coloneqq
	\rho
	\HAR(\rho)
	$
	be given and assume the function~$a()$ is twice continuously differentiable in a small neighbourhood around~$\HAR(\rho)$. 
	Then, the first-order correction to the local equilibrium approximation of the relaxed model
	\begin{equation}\label{SystemRelaxAwRascle}
		\begin{cases}
			\begin{aligned}
				\partial_t \rho + \partial_x\Big(\rho v(\rho,z,w) \Big)&=0,\\
				\partial_t z+\partial_x 
				\Big(
				zv(\rho,z,w)
				\Big)
				&=0, \\
				\partial_t w +  \partial_x \Big(wv(\rho,z,w)
				\Big)
				&=
				-
				\frac{w-\WAR(\rho)}{\varepsilon}
			\end{aligned}
		\end{cases}
	\end{equation}	
	for 
	$\displaystyle
	v(\rho,z,w)
	=
	\frac{z}{\rho}-a\Big(\frac{w}{\rho}\Big) P(\rho)
	$ 	reads as
	\begin{equation*}
		\begin{cases}
			\begin{aligned}
				\partial_t \rho + \partial_x \Big(
				\rho v\big(\rho,z,\WAR(\rho)\big)
				\Big)
				&= 
				\varepsilon
				\partial_x
				\Big(\rho
				\mu(\rho)
				\partial_x v\big(\rho,z,\WAR(\rho)\big)	\Big)
				+
				\mathcal{O}\big(
				\varepsilon^2
				\big), \\
				\partial_t z + 
				\partial_x
				\Big(
				zv\big(\rho,z,\WAR(\rho)\big)			\Big) 
				&= 
				\varepsilon
				\partial_x \Big(z
				\mu(\rho)
				\partial_x v\big(\rho,z,\WAR(\rho)\big)
				\Big)+
				\mathcal{O}\big(\varepsilon^2\big)
			\end{aligned}
		\end{cases}
	\end{equation*}
	with the velocity-gradient diffusion coefficient~$\mu(\rho)\coloneqq
	a'\big(\HAR(\rho)\big)
	\HAR'(\rho) P(\rho) \rho
	$. 
	It is negative, i.e.~$\mu(\rho)<0$, in the case~$
	\partial_\rho a\big(\HAR(\rho)\big)<0
	$ with~$\gamma>0$. 
	The first-order correction can be written as
	$$
	\partial_t \ueq + \partial_x \feq\big(\ueq\big)  
	=
	\varepsilon \partial_x
	\Big(
	D\big(\ueq \big) \partial_x \ueq \Big) 
	\quad
	\text{for}
	\quad
	\feq(\rho,z)=
	\begin{pmatrix}
		z-\rho Q(\rho) \\
		\frac{z^2}{\rho}-zQ(\rho)
	\end{pmatrix}
	$$
	where the diffusion matrix reads as
	\begin{align*}
		&D\big(\ueq \big) = \mu(\rho) 
		\begin{pmatrix} 
			-\big(\veq + Q(\rho) + q(\rho)\big) & 1 \\
			-\big(\veq + Q(\rho)\big)\big(\veq + Q(\rho) + q(\rho)\big) & \veq + Q(\rho) 
		\end{pmatrix} \\
		&
		\text{with  eigenvalues}
		\ \
		\lambda_1^{\textup{D}} = 0, \ \ \lambda_2^{\textup{D}}(\ueq) = \Big(
		\lambdaeq_1(\ueq) - \lambdaf_1(\ueq)
		\Big) q(\rho) 
		=
		-\mu(\rho)q(\rho)\\
		&
		\text{for }\quad
		Q(\rho)
		\coloneqq 
		a\big(
		\HAR(\rho)
		\big) P(\rho), \quad
		q(\rho)\coloneqq
		Q'(\rho) \rho \\
		&
		\text{and}\quad
		\lambdaeq_1(\rho,v)
		\coloneqq
		v
		-
		q(\rho), \quad \ \ \,
		\lambdaf_1(\rho,v) 
		\coloneqq
		\lambda_1\big(\rho,
		v,\HAR(\rho)	\big). 
	\end{align*}

\end{theorem}

\begin{proof}
	Lemma~\ref{Oldpaper} implies for the second equation the expression
	\begin{align*}
		&	\partial_t z
		+
		\partial_x \Big(
		\frac{z^2}{\rho}
		-
		za(h)P(\rho)
		\Big)\\
		=\
		&	\partial_t z
		+
		\partial_x \Big(
		\frac{z^2}{\rho}
		-
		za\big(\HAR(\rho)\big)
		P(\rho) \Big)  -
		\varepsilon
		\partial_x
		\Big(
		z\rho	a'\big(\HAR(\rho)\big)
		\HAR'(\rho) P(\rho)
		\partial_x v\big(\rho,z,\WAR(\rho)\big)	\Big)
		+
		\mathcal{O}\big(
		\varepsilon^2
		\big).
	\end{align*}
	Furthermore, Lemma~\ref{Oldpaper} yields
	\begin{equation*}    
		v(\rho,z,w)
		=\frac{z}{\rho}-
		a\big(\HAR(\rho)\big) P(\rho)
		- \varepsilon \rho
		a'\big(\HAR(\rho)\big)
		\HAR'(\rho)
		P(\rho)		\partial_x v\big(\rho,z,\WAR(\rho)\big)	+
		\mathcal{O}\big(\varepsilon^2\big).
	\end{equation*}
	Hence, we have for the first equation the expression
	\begin{align*}    
		0&=
		\partial_t \rho + \partial_x \Big(
		\rho v\big(\rho,z,w\big)
		\Big) \\
		&=
		\partial_t \rho + \partial_x \Big(
		\rho v\big(\rho,z,\WAR(\rho)\big)
		\Big)  -
		\varepsilon
		\partial_x
		\Big(\rho^2
		a'\big(\HAR(\rho)\big)
		\HAR'(\rho) P(\rho) 
		\partial_x v\big(\rho,z,\WAR(\rho)\big)	\Big)
		+
		\mathcal{O}\big(
		\varepsilon^2
		\big).
	\end{align*}   
	Then, the claim follows from~$z=\rho\big(
	v+a(h)P(\rho) \big)>0$.
\end{proof}

\medskip

\begin{remark}[Diffusion matrix]\label{RemarkDiffusion}	The diffusion matrix admits the rank-one dyadic factorisation 
	$
	D(\rho, v) = \boldsymbol{\ell}(\rho, v) \boldsymbol{r}^\T(\rho, v)
	$
with
$$
	\boldsymbol{\ell}(\rho, v) = \mu(\rho) 
	\begin{pmatrix}
		1 \\  v + Q(\rho)
	\end{pmatrix}
	\quad \text{and} \quad
	\boldsymbol{r}(\rho, v) = \begin{pmatrix} -v - Q(\rho) - q(\rho) \\  1 \end{pmatrix}. 
	$$
	Due to~$z^{\textup{eq}} 
	=
	\rho 
	\big(
	\veq + Q(\rho) 
	\big)$ and 
	$
	\partial_x z^{\textup{eq}}  = \big(v + Q(\rho) + q(\rho)\big) \partial_x \rho + \rho  \partial_x \veq 
	$, 
	the equality~$
	\boldsymbol{r}^\T(\rho, v) 
	\partial_x \ueq 
	=
	\rho \partial_x \veq
	$
	reveals an exact cancellation of the density gradient terms. 		Algebraically, any purely density-driven gradient field, i.e.~$\partial_x \rho \neq 0$ with $\partial_x v = 0$, lies strictly within the  nullspace of the diffusion matrix. 
	Hence, the diffusion matrix guarantees that the parabolic smoothing  restores equilibrium by suppressing spatial velocity fluctuations  without introducing artificial diffusion into the vehicle density~$\rho$ provided that the equalities~$
	\mu(\rho)<0 $ 
and
$$	\lambda_2^{\textup{D}}(\ueq) = \Big(
	\lambdaeq_1(\ueq) - \lambdaf_1(\ueq)
	\Big) q(\rho) \geq 0,
	\quad
	\text{i.e.}
	\quad
	\lambdaeq_1(\ueq) \geq \lambdaf_1(\ueq)
	$$ 
	hold, which is discussed in the following subsection.

\end{remark}

\medskip

\subsection{Frozen characteristic speeds and equilibrium system} \label{SectionFrozen}
On the equilibrium manifold~$h=\HAR(\rho)$
the homogeneous relaxation system, i.e.~the PIU model, has the \textbf{frozen} characteristic speeds
\begin{equation*}
	\lambda_1^{\textup{f}}\big(\rho,v\big)
	=
	\lambda_1\big(\rho,v,\HAR(\rho) \big)
	=
	v- a\big(\HAR(\rho)\big)P'(\rho)\rho,
	\quad
	\lambda_2^{\textup{f}}(v)=\lambda_3^{\textup{f}}(v)=\lambda_v(v)= v.
\end{equation*}
In the relaxation limit, 
the equilibrium flux~$\feq$ has  the equilibrium characteristic speeds 
\begin{equation*}
	\feq(\rho,z)
	=
	\begin{pmatrix}
		z-\rho Q(\rho) \\
		\frac{z^2}{\rho}-zQ(\rho)
	\end{pmatrix}
	\quad\text{with}\quad
	\lambda_1^{\textup{eq}}(\rho,v)
	=
	v-\rho Q'(\rho), \quad
	\lambda_2^{\textup{eq}}(v)
	=
v.
\end{equation*}
Since we have the relation~$
\lambda_1^{\textup{eq}}(\rho,v)
=
v- a\big(\HAR(\rho)\big)P'(\rho)\rho
-\mu(\rho)
=
\lambdaf_1 (\rho,v)
-
\mu(\rho)
$ for the coefficient $
\mu(\rho)
= a'\big(\HAR(\rho)\big)\HAR'(\rho)P(\rho)\rho
<0
$ in the case $
\partial_\rho a\big(\HAR(\rho)\big)<0$,      
Theorem~\ref{TheoremAwRascle1} satisfies the inequalities
$$
\lambdaf_1 (\rho,v)
\leq 
\lambda_1^{\textup{eq}}(\rho,v)
\leq 
\lambdaf_2(v)=\lambdaf_3(v)=v. 
$$
This  is the appropriate subcharacteristic ordering which guarantees that the equilibrium characteristic cone is contained in the characteristic cone of the full relaxation system. 
In particular,  the characteristic cone of the equilibrium system is
contained in that of the full relaxation system, which is the 
appropriate subcharacteristic condition for stability. 
Lemma~\ref{Oldpaper} implies the relation 
\begin{align*}
		v
&=
\veq(\rho)
-\varepsilon
a'\big(\HAR(\rho)\big)\HAR'(\rho)P(\rho)\rho
\partial_x 
\veq(\rho)
+\mathcal{O}\big(\varepsilon^2\big) \\
&=
\veq(\rho)
-\varepsilon\mu(\rho)\partial_x 		\veq(\rho)
+\mathcal{O}\big(\varepsilon^2\big), 
\end{align*}
where~$
\mu(\rho)
=
a'\big(\HAR(\rho)\big)
\HAR'(\rho)P(\rho)\rho<0
$ accounts for positive diffusion in Theorem~\ref{TheoremAwRascle1}.  This means that 
the finite relaxation of the driver-dependent property produces a gradient correction with respect to spatial variations of the equilibrium velocity that arise from the term~$\partial_x \veq$.  The following Section~\ref{SectionRelaxNonCons}  analyzes these variations.

\medskip
\begin{remark}[Frozen and equilibrium velocity]
	The frozen velocity refers to the velocity at the state where the system is frozen for the analysis of its characteristic speeds. If the system is frozen at an equilibrium state, i.e.~for
	$
	h=H(\rho)
	$, 
	the frozen velocity coincides with the equilibrium velocity~$
	v^{\textup{f}}=\veq(\rho,z)
	=\nicefrac{z}{\rho}-a\big(\HAR(\rho)\big)P(\rho).
	$
	Consequently, the frozen characteristic velocities evaluated at an equilibrium state can be expressed in terms of the equilibrium velocity. In particular, it holds 
	$
	\lambda_2^{\textup{f}}(v)=\lambda_3^{\textup{f}}(v)=v{^\textup{f}}=\veq.
	$ 
	Outside the equilibrium manifold, the frozen velocity does not in general coincide with the equilibrium velocity and we have
	$$
	v^{\textup{f}}
	=
	v(\rho,z,w)
	=
	\frac{z}{\rho}-a\Big(\frac{w}{\rho}\Big)P(\rho)
	\ \neq \
	\frac{z}{\rho}-a\big(\HAR(\rho)\big)P(\rho)
	=\veq(\rho,z). 
	$$

\end{remark}

\subsection{Non-conservative limit} \label{SectionRelaxNonCons}
We now introduce the relaxation mechanism directly at the microscopic
level and verify that its hydrodynamic limit yields the relaxed
macroscopic system in Theorem~\ref{TheoremAwRascle1}.  
On the microscopic level, drivers are  allowed to adjust their acceleration behaviour according  to the relaxation
\begin{equation}\label{hRelaxationMicro}
	\dot h_i
	= -
	\frac{h_i-\HAR(\rho_i)}{\varepsilon}.
\end{equation}
Since the quantity $
\omega_i=v_i+a(h_i)P(\rho_i)
$ remains invariant,~i.e.~$\dot{\omega}=0$, it holds
\begin{align}
	&0
	=
	\dot{v}_i
	+a'(h_i)P(\rho_i)\dot{h}_i
	+a(h_i)P'(\rho_i)\dot{\rho}_i
	\quad
	\text{for}
	\quad
	\dot{\rho}_i
	=
	-\rho_i
	\frac{v_{i+1}-v_i}{x_{i+1}-x_i}\nonumber\\
	&\Leftrightarrow
	\qquad
	\dot{v}_i
	=
	a(h_i)P'(\rho_i)\rho_i
	\frac{v_{i+1}-v_i}{x_{i+1}-x_i}
	+
	a'(h_i)P(\rho_i)
	\frac{h_i-\HAR(\rho_i)}{\varepsilon}.\label{velocityRelaxation}
\end{align}
In the hydrodynamic limit~the invariant property~$\dot{\omega}_i=0$ still implies the second equation in the  conservative formulation~\eqref{PiuS3}, i.e.~$
\partial_t z + \partial_x (v z) =0$.  The remaining equation of the relaxation system in Theorem~\ref{TheoremAwRascle1} 
follows from the identity~$
\dot{w}=
\partial_t w + v \partial_x w 
$, $h\dot{\rho}= -w \partial_x v$ 
and
\begin{equation}
	\label{SecReltemp1}
	\partial_t w 
	+
	\partial_x(vw)
	=
	\dot{w}
	+
	w \partial_x v
	=
	h \dot{\rho}
	+
	\rho\dot{h}
	+w \partial_x v 
	=
	\rho \dot{h}
	=-
	\frac{w-\WAR(\rho)}{\varepsilon}.
\end{equation}
Note that we passed from the microscopic relaxation~\eqref{hRelaxationMicro} to the macroscopic one in Theorem~\ref{TheoremAwRascle1} over the invariants, which result in a conservative form in Lagrangian coordinates. Hence, we circumvented the equation~\eqref{velocityRelaxation} for the velocity. Without a formal proof, this equation indicates that the relaxation~\eqref{hRelaxationMicro} comes along with a relaxation in the velocity, namely~
\begin{equation*}
	\partial_t v+ \Big(v-a(h)P'(\rho)\rho \Big)\partial_x v
	=
	a'(h)P(\rho)
	\frac{h-\HAR(\rho)}{\varepsilon}.
\end{equation*}
Indeed, we obtain the expression also formally by rewriting the 
macroscopic equations in 
Theorem~\ref{TheoremAwRascle1}. 
Due to equation~\eqref{SecReltemp1},  the relation~$\dot{h}
=-
\frac{1}{\varepsilon}
\big(h-
\HAR(\rho)
\big)
$ 
is satisfied 
and hence it holds
\begin{align*}
	\partial_t v
+
v
\partial_x v 
&=
\dot{v}=
-a(h)P'(\rho)\dot{\rho}
-a'(h)P(\rho) \dot{h} \\
&=
a(h)P'(\rho)\rho \partial_xv
+a'(h)P(\rho)\frac{h-\HAR(\rho)}{\varepsilon}.
\end{align*}


\begin{lemma}[Second-order Chapman-Enskog expansion]
	\label{LemmaSecondOrder}
	Let a local equilibrium ${\HAR\in C^3\big(\mathbb{R}^+\big)}$ be given and assume that the function $a\in C^2$ is twice continuously differentiable in a neighbourhood of $\HAR(\rho)$. Then, the second-order local equilibrium approximation reads as
	\begin{align*}
		h=\HAR
		&+\varepsilon\rho \HAR'(\rho)\alpha
		+\varepsilon^2
		\Big(
		\rho\big(2\HAR'(\rho)+\rho \HAR''(\rho)\big)\alpha^2
		-\rho \HAR'(\rho) q(\rho)\partial_x\alpha\\
		&
		-\rho \HAR'(\rho)q'(\rho)\alpha\,\partial_x\rho
		+a'\big(\HAR(\rho)\big)P(\rho)\rho \HAR'(\rho)^2
		\alpha\,\partial_x\rho
		\Big)
		+\mathcal{O}\big(\varepsilon^3\big),
	\end{align*}
	where the 
	spatial derivative of the equilibrium velocity, namely~$
	\alpha=\partial_x\veq, 
	$ is given by the  equilibria $\veq(\rho,z) = v\big(\rho,z,\WAR(\rho)\big)$ and  $\WAR(\rho)=\rho\HAR(\rho)$.
\end{lemma}

\medskip

\begin{proof}
	Following the Chapman-Enskog expansion  in Lemma~\ref{Oldpaper}, we write
	$$
	h=\HAR(\rho)+\varepsilon h_1+\varepsilon^2h_2
	+O\big(\varepsilon^3\big)	
	=
	\HAR(\rho)+\varepsilon
	\HAR'(\rho) \rho \partial_x \veq(\rho,z)
	+\varepsilon^2h_2
	+O\big(\varepsilon^3\big).
	$$
	At the next order, the relaxation equation gives
	\begin{equation}
		\label{Lemmah2}
		-
		h_2
		=
		\partial_t\big(\rho \HAR'(\rho)\alpha\big)
		+
		\veq(\rho,z) \partial_x
		\big(\rho \HAR'(\rho)\alpha\big)\\
		-
		a'\big(\HAR(\rho)\big)
		P(\rho) \rho  \HAR'(\rho)^2 \alpha \partial_x \rho.
	\end{equation}
	Differentiating the equilibrium velocity equation 
	$$
	\partial_t\veq
	+
	\Big(
	\veq-q(\rho)
	\Big)
	\partial_x \veq
	=
	0
	\ \
	\text{ for }
	\ \
	q(\rho) = B(\rho) + \mu(\rho),\ \
	B(\rho)= a\big(\HAR(\rho)\big)P'(\rho)\rho
	$$
	with respect to space~$x$ gives
	$$
	\partial_t\alpha
	+\veq
	\partial_x\alpha
	=
	q(\rho)\alpha_x
	+q'(\rho) \alpha
	\rho_x
	-\alpha^2.
	$$
	The claim follows from conservation of mass~$
	\partial_t\rho+v_{\mathrm{eq}}\partial_x\rho
	=-\rho\alpha
	$, which implies the identity
	$
	\partial_t\big(\rho \HAR'(\rho)\big)
	+\veq \partial_x\big(\rho \HAR'(\rho)\big)
	=
	-\rho\big(\HAR'(\rho)+\rho \HAR''(\rho)\big)\alpha
	$ 
	and~equation~\eqref{Lemmah2}, which yields
	\begin{align*}
		&				h_2
		=
		\rho 
		\Big(2\HAR'(\rho)+\rho\HAR''(\rho)
		\Big)\alpha^2
		-\rho \HAR'(\rho)q(\rho)\partial_x\alpha 
		-\rho\HAR'(\rho)q'(\rho)\alpha\partial_x\rho 				+a'\big(\HAR(\rho)\big)
		P(\rho) \rho  \HAR'(\rho)^2 \alpha \partial_x \rho.			
	\end{align*}
\end{proof}

\begin{theorem}[Viscous Hamilton-Jacobi system]\label{TheoremAwRascle2}
	Suppose an equilibrium~$
	\HAR(\rho)
	$ 
	is given and assume the function~$a$ is twice continuously differentiable in a small neighbourhood around~$\HAR(\rho)$. 
	Then, the first-order correction to the local equilibrium approximation of the relaxation
	\begin{equation*}
		\begin{cases}
			\begin{aligned}
				\partial_t \rho + \partial_x(\rho v)&=0,\\
				\partial_t v+ \Big(v-a(h)P'(\rho)\rho\Big)\partial_x v
				&= a'(h)P(\rho)\frac{h-\HAR(\rho)}{\varepsilon}, \\
				\partial_t h 
				+
				v\partial_x h
				&=-
				\frac{h-\HAR(\rho)}{\varepsilon}
			\end{aligned}
		\end{cases}
	\end{equation*}
	reads as
	\begin{align*}
		&\partial_{t}v + \Big(v - q(\rho)\Big)\partial_{x}v 
		=
		\varepsilon \mu(\rho)
		\Big[	
		\underbrace{C(\rho) (\partial_x v)^2 }_{\text{Hamilton-Jacobi}}
		-
		\underbrace{M(\rho)^{-1} \partial_x \big( M(\rho) B(\rho) \partial_x v \big) }_{\text{diffusion}} 
		\Big]	\\[-14pt] 
		&\hspace{3cm} + \mathcal{O}\big(\varepsilon^2\big),\\[4pt]
		&\widetilde{C}(\rho)\coloneqq
		\frac{2\HAR'(\rho)+\rho \HAR''(\rho)}{\HAR'(\rho)}
		+\rho \HAR'(\rho)
		\frac{a''\big(\HAR(\rho)\big)}{a'\big(\HAR(\rho)\big)}, \quad
		C(\rho)
		\coloneqq \widetilde{C}(\rho) + 
		\rho \frac{P'(\rho)}{P(\rho)}, \\
 &{B(\rho) \coloneqq q(\rho)-\mu(\rho)}, \quad
		{M(\rho) \coloneqq \exp \left( - \int \frac{\mu(\rho)}{\rho B(\rho)}  \d\rho \right)}\\
&\text{with}\quad	q(\rho)
= 
a\big(\HAR(\rho)\big)P'(\rho)\rho
+
\mu(\rho).
	\end{align*}
\end{theorem}

\begin{proof}
	Lemma~\ref{LemmaSecondOrder} yields 
	\begin{equation*}
		\begin{aligned}
			h=\HAR
			&+\varepsilon\rho \HAR'(\rho)\alpha
			+\varepsilon^2
			\Big(
			\rho\big(2\HAR'(\rho)+\rho \HAR''(\rho)\big)\alpha^2
			-\rho \HAR'(\rho) q(\rho)\partial_x\alpha\\
			&
			-\rho \HAR'(\rho)q'(\rho)\alpha\,\partial_x\rho
			+a'\big(\HAR(\rho)\big)P(\rho)\rho \HAR'(\rho)^2
			\alpha\,\partial_x\rho
			\Big)
			+\mathcal{O}\big(\varepsilon^3\big).
		\end{aligned}
	\end{equation*}
	Hence, the relaxation term satisfies
	\begin{equation*}	
		\begin{aligned}
			& a'(h)P(\rho)\frac{h-\HAR(\rho)}{\varepsilon} \nonumber\\
			&\quad
			=\mu(\rho)\,\alpha
			+\varepsilon\mu(\rho)
			\Big[
			\widetilde{C}(\rho)\alpha^2
			-\Big(
			q'(\rho)-\frac{\mu(\rho)}{\rho}
			\Big)\alpha\,\partial_x\rho
			-q(\rho)\,\partial_x\alpha
			\Big]
			+\mathcal{O}\big(\varepsilon^2\big).
		\end{aligned}
	\end{equation*}
	By Lemma~\ref{Oldpaper}, the velocity satisfies
	\begin{alignat}{5}
		&\quad v
		&&=
		\veq
		&&-\varepsilon\mu(\rho)\alpha
		+\mathcal{O}\big(\varepsilon^2\big)
		\quad\text{for}\quad
		\alpha=\partial_x\veq \nonumber \\
		\Rightarrow\quad
		&				\partial_xv
		&&=
		\alpha
		&&-\varepsilon
		\Big(
		\mu'(\rho)\partial_x\rho\,\alpha
		+\mu(\rho)\partial_x\alpha
		\Big)
		+\mathcal{O}\big(\varepsilon^2\big).\label{alpha-v}	
	\end{alignat}
	In particular, in all terms already multiplied by the relaxation parameter $\varepsilon>0$, we have the asymptotic relation~$	\alpha=\partial_xv+\mathcal{O}(\varepsilon)$ and $
	\partial_x\alpha=\partial_x^2v+\mathcal{O}(\varepsilon)$. 
	The identity~$
	q(\rho)=B(\rho)+\mu(\rho)$ 
	 implies the equations~$\mu'(\rho)-q'(\rho)=-B'(\rho)$ and 
	$
	\mu(\rho)-q(\rho)=-B(\rho)
	$. We apply  the second-order expansion~\eqref{alpha-v} in the leading-order term 	$\mu(\rho)\alpha$  to obtain
	\begin{alignat*}{8}
		& a'(h)P(\rho)\frac{h-\HAR(\rho)}{\varepsilon}
		&&=
		&&\mu(\rho)\partial_xv
		+\varepsilon\mu(\rho)
		\bigg[
		\widetilde{C}(\rho)(\partial_xv)^2
		+\mu'(\rho)\partial_x\rho\,\partial_xv
		+\mu(\rho)\partial_x^2v\\
		& && && \qquad  
		-\Big(
		q'(\rho)-\frac{\mu(\rho)}{\rho}
		\Big)\partial_xv\,\partial_x\rho
		-q(\rho)\partial_x^2v
		\bigg]
		&& \hspace{-1.2cm}
		+\mathcal{O}\big(\varepsilon^2\big) \\
		& && =
		&&\mu(\rho)\partial_xv 
		+\varepsilon\mu(\rho)
		\bigg[
		\widetilde{C}(\rho)(\partial_xv)^2\\
		& && && \qquad  
		-\Big(
		B'(\rho)-\frac{\mu(\rho)}{\rho}
		\Big)
		\partial_xv\,\partial_x\rho
		-B(\rho)\partial_x^2v
		\bigg]
		&&\hspace{-1.2cm}
		+\mathcal{O}\big(\varepsilon^2\big).
	\end{alignat*} 
	We define 
the auxiliary function 
		$
	\phi(\rho)
	=
	a'\big(\HAR(\rho)\big)
	\HAR'(\rho)
	P'(\rho)\rho^2$. 
	Due to~Lemma~\ref{Oldpaper}, the left-hand side satisfies
	\begin{alignat*}{8}
&		\partial_t v
		+\Big(v- \quad a(h)&&P'(\rho)\rho\Big)\partial_xv \\
		=
&		\partial_t v
		+\Big(
		v-a\big(\HAR(\rho)\big)&&P'(\rho)\rho
		\Big)\partial_xv
		-\varepsilon
		\phi(\rho)(\partial_xv)^2
		+\mathcal{O}\big(\varepsilon^2\big).
	\end{alignat*}
Using  the expression ${
	q(\rho)
	=
	a\big(\HAR(\rho)\big)P'(\rho)\rho+\mu(\rho)
}$, 
	we obtain
	\begin{align*}
		&	\partial_t v
		+\big(v-q(\rho)\big)\partial_xv \\
&		=
		\varepsilon\mu(\rho)
		\left[
		C(\rho)(\partial_xv)^2
		-\left(
		B'(\rho)-\frac{\mu(\rho)}{\rho}
		\right)
		\partial_xv\,\partial_x\rho
		-B(\rho)\partial_x^2v
		\right]
		+\mathcal{O}\big(\varepsilon^2\big) \\
		&
		\text{with}\quad
		\frac{\phi(\rho)}{\mu(\rho)}
		=
		\rho\frac{P'(\rho)}{P(\rho)}
		\quad\text{and}\quad
		C(\rho)
		=
		\widetilde{C}(\rho)
		+\rho\frac{P'(\rho)}{P(\rho)}.
	\end{align*}
	The proof follows from the definition of $M(\rho)$, which yields the identity
	$$			-\left(
	B'(\rho)-\frac{\mu(\rho)}{\rho}
	\right)
	\partial_xv\,\partial_x\rho
	-B(\rho)\partial_x^2v
	=
	-\frac{1}{M(\rho)}
	\partial_x
	\big(
	M(\rho)B(\rho)\partial_xv
	\big).
	$$
\end{proof}
Firstly, we directly observe  from~Theorem~\ref{TheoremAwRascle2} that positive diffusion is added in the case~$\mu(\rho)<0$ 
which is in accordance with the discussions in Remark~\ref{RemarkDiffusion} and Subsection~\ref{SectionFrozen}. The following corollary summarizes this important statement. 

\medskip

\begin{corollary}[Diffusion]\label{CorollaryStable}
	If the velocity-gradient diffusion coefficient~$\mu(\rho)=
	a'\big(\HAR(\rho)\big)
	\HAR'(\rho) P(\rho) \rho
	<0$ is negative, i.e.~the inequality
	$  		\partial_\rho a\big(\HAR(\rho)\big)<0$   
	holds, both Theorem~\ref{TheoremAwRascle1} and Theorem~\ref{TheoremAwRascle2} guarantee positive diffusion that is added to the velocity in the small relaxation limit.
\end{corollary}

\medskip

To gain further physical insight into the small relaxation limit, we decompose the right-hand side  of the local equilibrium approximation into a Hamilton-Jacobi and a  diffusion term. 
More precisely, the diffusion is a pure dissipative mechanism that smooths small-scale perturbations without affecting wave propagation speeds. 

In contrast, the term  $(\partial_x v)^2$  acts as a Hamilton-Jacobi correction that modifies the characteristic wave speeds. 
Neglecting the parabolic diffusion term and writing for short~$
\alpha = \partial_x \veq
$ in equilibrium, i.e.~$v=\veq$, 
the system reduces to the quasilinear first-order system 
\begin{equation}
	\label{HJ1}
	\partial_{t}v + \Big(v - q(\rho)
	-\varepsilon
	\mu(\rho)
	C(\rho) \alpha
	\Big)\partial_{x}v
	=0. 
\end{equation}
Differentiating equation~\eqref{HJ1} with respect to the spatial variable~$x$, we obtain
\begin{equation}\label{HJ2}
	\begin{aligned}
		0&=	\partial_{t}\alpha + \Big(v - q(\rho)
		-\varepsilon
		\mu(\rho)
		C(\rho) \alpha
		\Big)\partial_{x}\alpha  \\
		&+
		\alpha \partial_x v
		-
		\alpha
		q'(\rho) \partial_x \rho
		-
		\varepsilon
		\Big(
		\partial_\rho\big(
		\mu(\rho) C(\rho)
		\big) 
		\alpha^2
		\partial_x \rho 
		+
		\mu(\rho) C(\rho) \alpha 
		\partial_x \alpha
		\Big).
	\end{aligned}
\end{equation} 
Introducing the state vector
$
\widetilde{\u}
=(\rho,v,\alpha)^{\T}$, 
equations~\eqref{HJ1} and~\eqref{HJ2} 
can be written in quasilinear form as
$
\partial_t \widetilde{\u} + 
A^{\textup{HJ}}
\big(\widetilde{\u}\big)\partial_x \widetilde{\u}=0,
$
where the matrix
$$
A^{\textup{HJ}}
\big(\widetilde{\u}\big)
=
\begingroup
\setlength{\arraycolsep}{-8pt} 
\begin{pmatrix}
	v & \rho & 
	\\
	 &
	v-q(\rho)-\varepsilon\mu(\rho)C(\rho)\alpha
	& 
	\\
	-\alpha q'(\rho)-\varepsilon(\mu C)'(\rho)\alpha^2
	& \alpha
	&
	v-q(\rho)-2\varepsilon\mu(\rho)C(\rho)\alpha
\end{pmatrix}
\endgroup
$$
yields the characteristic speeds
$$
\lambda_1^{\textup{HJ}}(\rho,v)
=
v-q(\rho)-\varepsilon\mu(\rho)C(\rho)\alpha,
\quad
\lambda_2^{\textup{HJ}}(\rho,v)
=
v-q(\rho)-2\varepsilon\mu(\rho)C(\rho)\alpha
$$
and~$\lambda_v(v)=v$. 
The important message from these physical considerations in terms of the characteristic speeds of the Hamilton-Jacobi term, which is deduced in Theorem~\ref{TheoremAwRascle2}, is the property
\begin{equation}\label{ImportantCondition}
	\textup{sign}\Big\{
	\lambda_1^{\textup{HJ}}(\rho,v) 
	-
	\lambda_1^{\textup{eq}}(\rho,v)
	\Big\}
	=
	\textup{sign}\Big\{
	\partial_x \veq
	\Big\}
	\quad
	\text{if}
	\quad
	\mu(\rho)C(\rho)<0
\end{equation}
holds, which is satisfied in the standard case~$\mu(\rho)<0$ and $C(\rho)>0$  that is ensured for small densities in the following Corollary, and provided that the velocity~$v=\veq+\mathcal{O}(\varepsilon)$ is sufficiently close to the equilibrium. 	
\medskip

\begin{corollary}[Hamilton-Jacobi correction in light traffic]\label{CorollaryCpos}
	Consider an equilibrium velocity 
	$V(\rho)>0$ 
	with negative derivative~$
	V'(\rho)<0$ and a finite vacuum limit~$ 
	\lim_{\rho\searrow 0^+}V(\rho)\le\omega_0.
	$ 
	Let $P(\rho)$ be any strictly increasing pressure law, 
	let the function $\HAR(\rho)$ and the PIU closure $a(h)$ be related to the velocity $V(\rho)$ and the pressure law $P(\rho)$ through the equilibrium relation
	$$
	a\big(\HAR(\rho)\big)P(\rho)=\omega_0-V(\rho).
	$$
	Furthermore, we assume that the second derivative $a''$ exists and is continuous in a neighbourhood of $H_0$. 
	In particular, the condition~\eqref{LemmaCpos1} reduces in the special case~\eqref{PIUconservative} to
	$
	\lim\limits_{\rho\searrow 0^+}
	\rho
	\HAR''(\rho)\HAR'(\rho)^{-1}
	>-\gamma-2.
	$ 
	
	Then, the inequality $C(\rho)>0$ holds for all sufficiently small densities
	$\rho \in (0,\rho^*)$ with  some critical density~$\rho^*>0$  provided that the conditions
	\begin{align}
		&\lim\limits_{\rho\searrow 0^+}V(\rho)\leq\omega_0<\infty,
		\quad
		\lim\limits_{\rho\searrow 0^+}
		a'\big(\HAR(\rho)\big) 
		=
		a'\big(H_0\big)
		\neq 0, 
		\quad
		\lim\limits_{\rho\searrow 0^+}
		\rho \HAR'(\rho) =0 \nonumber \\ 
		& \text{and}\quad
		\lim\limits_{\rho\searrow 0^+}
		\left[
		\rho\frac{P'(\rho)}{P(\rho)}
		+
		\rho\frac{\HAR''(\rho)}{\HAR'(\rho)}
		\right]>-2 \label{LemmaCpos1}
	\end{align}
	hold close to  vacuum states. 
	
\end{corollary}

\begin{proof}
	The assumptions yield
	$$
	\rho \HAR'(\rho)
	\frac{a''\big(\HAR(\rho)\big)}{a'\big(\HAR(\rho)\big)}
	\rightarrow 0
	\quad
	\Rightarrow
	\quad
	C(\rho)
	=
	2+
	\rho\frac{P'(\rho)}{P(\rho)}
	+
	\rho\frac{\HAR''(\rho)}{\HAR'(\rho)}
	+{o}(1)
\ \
	\text{for}
\ \
	\rho \searrow 0^+.
	$$
	Therefore the term $C(\rho)>0$ is strictly positive for all sufficiently small densities $\rho>0$ provided that the inequality~\eqref{LemmaCpos1} holds.
\end{proof}
We conclude that, in the case of light traffic, the Hamiltonian
\[
\mathcal{H}(\rho,v,p)
=
\bigl(v-q(\rho)\bigr)p
-
\varepsilon\mu(\rho)C(\rho)p^2
\quad\text{with}\quad
 p=\partial_x v
\]
in Theorem~\ref{TheoremAwRascle2} is convex with respect to
the gradient variable \(p\). Moreover, the diffusive term in
the viscous Hamilton-Jacobi equation provides additional
regularization and has a stabilizing effect. 
This convexity is an important structural property in the theory of
viscosity solutions, as it supports comparison principles and, under
suitable assumptions, ensures uniqueness and stability of solutions.
Furthermore, the vanishing-viscosity limit provides a natural way
to select the viscosity solution of the corresponding first-order
Hamilton-Jacobi equation~\cite{Crandall1983,CocliteRisebro2007,CocliteKarlsen2012}.

\section{Calibration of equilibria}
\label{SectionCalibration}
We state explicit choices of the equilibrium  such that   condition~\eqref{ImportantCondition}, namely~$\mu(\rho)C(\rho)<0$, is satisfied. 
Its significance is explained by   the propagation of traffic waves in the Hamilton-Jacobi limit. The finite relaxation time of the drivers introduces a correction to the characteristic speed, which is given by 
$\big(\lambda_1^{\textup{HJ}}-\lambda_1^{\textup{eq}}\big)
(\rho,v)
=
-\varepsilon\mu(\rho)C(\rho)\partial_x\veq
$ 
for~$(v-\veq)=\mathcal{O}(\varepsilon)$. 
Since the modification of the wave speed induced by finite relaxation is determined by the spatial variation of the velocity, we have the implications
$$
\begin{aligned}
	\partial_x \veq > 0
	&\quad\Rightarrow\quad
	\lambda_1^{\textup{HJ}}(\rho,v)
	\geq
	\lambda_1^{\textup{eq}}(\rho,v),
	\\
	\partial_x \veq < 0
	&\quad\Rightarrow\quad
	\lambda_1^{\textup{HJ}}(\rho,v)
	\leq
	\lambda_1^{\textup{eq}}(\rho,v).
\end{aligned}
$$
Hence, the propagation speed of traffic flows is always logically coupled to the current acceleration and braking behavior of the simulated vehicles. The following examples ensure  consistency with typically used fundamental diagrams, while simultaneously providing the desired stabilizing effect on the velocity that is guaranteed according to Corollary~\ref{CorollaryStable} in the case of light traffic.

\medskip

\begin{example}[Power-law closure]
	Consider the traffic pressure~$
	P(\rho)=\nicefrac{\rho^\gamma}{\gamma} 
	$ 
	and	let the closure $a\in C^2$ be strictly monotone on the relevant range such that its inverse~$a^{-1}$ is well-defined. Furthermore,  we define
	$$
	R(\rho)=a\big(\HAR(\rho)\big)
	=R_0-k\Big(\frac{\rho}{\rho_{\max}}\Big)^m
	\quad
	\text{for}
	\quad
	\quad k>0,\quad m>0.
	$$
	The corresponding equilibrium reads as
	$$
	V(\rho)
	=
	\omega_0-a\big(\HAR(\rho)\big)P(\rho)
	=
	\omega_0
	-
	\left[
	R_0-k\Big(\frac{\rho}{\rho_{\max}}\Big)^m
	\right]
	\frac{\rho^\gamma}{\gamma}.
	$$
	The term~$C(\rho)>0$ is strictly positive for all admissible densities $\rho\in(0,\rho_{\max}]$ as it holds
	$$
	\rho\frac{P'(\rho)}{P(\rho)}=\gamma,
	\quad
	\rho\frac{R''(\rho)}{R'(\rho)}
	=m-1
	\quad
	\Rightarrow
	\quad
	C(\rho)=\gamma+m+1>0. 
	$$ 
	In the case~$\gamma>0$ the velocity-gradient diffusion coefficient~$\mu(\rho)<0$ is negative, 
	since the relation~$	\mu(\rho)
	=
	R'(\rho)P(\rho)\rho 
	$ holds. Hence, the diffusion has a stabilizing effect, which is proven in Theorems~\ref{TheoremAwRascle1} and \ref{TheoremAwRascle2}.

\end{example}

\medskip

\begin{example}[Pipes-Munjal closure]
	Consider the Pipes-Munjal velocity
	\begin{equation*}
		V(\rho)
		=
		\omega_0
		\left(
		1-\Big(\frac{\rho}{\rho_{\max}}\Big)^n
		\right)
		\quad\text{for}\quad
		\gamma\neq n>0
		\quad\text{and}\quad
		P(\rho)=\frac{\rho^\gamma}{\gamma}.
	\end{equation*}
	The relation
	$	V(\rho)=\omega_0-a\big(\HAR(\rho)\big)P(\rho)
	$ yields the equalities
	\begin{equation*}
		R(\rho)=a\big(\HAR(\rho)\big)
		=
		\frac{\gamma \omega_0}{\rho_{\max}^n}
		\rho^{n-\gamma},
		\quad
		\frac{R''(\rho)}{R'(\rho)}
		=
		\frac{n-\gamma-1}{\rho}, 
		\quad
		\rho\frac{P'(\rho)}{P(\rho)}=\gamma. 
	\end{equation*}
	Therefore, the desired property~$C(\rho)>0$ holds for all admissible densities $\rho\in(0,\rho_{\max}]$, as it holds $C(\rho)=n+1$.
	Due to the relation~$	\mu(\rho)
	=
	R'(\rho)P(\rho)\rho 
	$ we have a negative coefficient, i.e.~$\mu(\rho)<0$ if and only if additionally the ordering $n<\gamma$ holds. In particular, for $n=1$ we have the Greenshields velocity law.
\end{example}

\medskip

\begin{example}[Kinetic equilibrium velocity]
	We consider the equilibrium speed arising from the kinetic traffic model
	of Tosin and Zanella~\cite{Zanella2019}. In normalized variables, the microscopic
	acceleration probability is
	$
	\Pi(\rho)=(1-\rho)^\nu,
	$ 
	and the asymptotic mean velocity of the kinetic model is
	\[
	V_\infty(\rho)
	=
	\frac{\Pi(\rho)}
	{\Pi(\rho)+\big(1-\Pi(\rho)\big)^2}.
	\]
	For the particular choice \(\nu=1\), and after rescaling the free-flow
	velocity by \(\omega_0>0\), this yields
	\[
	V(\rho)
	=
	\omega_0
	\frac{1-\rho}{1-\rho+\rho^2},
	\qquad
	\rho\in[0,1].
	\]
	In particular, the properties 
	$
	V(0)=\omega_0$ and  
	$ 
	V(1)=0
	$ hold. 
	Let the PIU pressure law be~$
	P(\rho)=\nicefrac{\rho^3}{3}.
	$
	The equilibrium relation
	yields 
	\begin{equation*}		
	R(\rho)=a\big(\HAR(\rho)\big)
	=
	\frac{3\omega_0}
	{\rho(1-\rho+\rho^2)} \quad
	\Rightarrow
	\quad
	\mu(\rho)
	=
	\rho P(\rho)R'(\rho)
	=
	-\omega_0
	\frac{
		\rho^2(3\rho^2-2\rho+1)
	}{
		(1-\rho+\rho^2)^2
	}<0
		\end{equation*}
	for every density \(\rho\in(0,1]\). 
	Moreover, the term
	$
	C(\rho)
	$
	has the explicit form
	\[
	C(\rho)
	=
	3
	\frac{
		\rho^4-3\rho^3+6\rho^2-3\rho+1
	}{
		(\rho^2-\rho+1)(3\rho^2-2\rho+1)
	}.
	\]
	All factors in the denominator are strictly positive and the numerator
	is positive for all densities. Hence,  the desired properties 
	$
	C(\rho)>0$ and 
	$\mu(\rho)C(\rho)<0
	$ hold on the whole density interval. 
\end{example}

\medskip

\begin{example}[Newell-Franklin closure]
	Consider the Newell-Franklin equilibrium velocity function
	\[
	V(\rho)
	=
	\omega_0
	\left[
	1-
	\exp\bigg(
	-\beta\Big(\frac{\rho_{\max}}{\rho}-1\Big)
	\bigg)
	\right]
	\quad
	\text{with}
	\quad
	\beta>0
	\]
	and the pressure law
	$
	P(\rho)=\frac{\rho^\gamma}{\gamma},
	\
	\gamma>0.
	$
	The particular equilibrium relation
	$
	V(\rho)=\omega_0-a\big(\HAR(\rho)\big)P(\rho)
	$
	yields
	\begin{align*}
		R(\rho)&=a\big(\HAR(\rho)\big)
		=
		\gamma\omega_0
		\rho^{-\gamma}
		\exp\bigg(
		-\beta\Big(\frac{\rho_{\max}}{\rho}-1\Big)
		\bigg)
\; \
		\Rightarrow
\; \
		\frac{R'(\rho)}{R(\rho)}
		=
		\frac{\beta\rho_{\max}-\gamma\rho}{\rho^2},\\
		\mu(\rho)
		&	=
		\rho P(\rho)R'(\rho)
		=
		\omega_0
		\exp\bigg(
		-\beta\Big(\frac{\rho_{\max}}{\rho}-1\Big)
		\bigg)
		\frac{\beta\rho_{\max}-\gamma\rho}{\rho}.
	\end{align*}
	Therefore, we obtain the two regimes
	\[
	\begin{cases}
		\mu(\rho)>0
		&\text{for}\quad
		0<\rho<\frac{\beta}{\gamma}\rho_{\max},
		\\
		\mu(\rho)<0
		&\text{for}\quad
		\rho>\frac{\beta}{\gamma}\rho_{\max}.
	\end{cases}
	\]
	On every interval  where it holds \(R'(\rho)\neq0\), we further obtain
	\begin{align*}    
		C(\rho)
		&	=
		2+
		\rho\frac{R''(\rho)}{R'(\rho)}
		+
		\rho\frac{P'(\rho)}{P(\rho)}
		=
		\frac{\beta\rho_{\max}}{\rho}
		-
		\frac{\gamma\rho}
		{\beta\rho_{\max}-\gamma\rho}, \\
		\mu(\rho)C(\rho)
		&	=
		\omega_0
		\exp\bigg(
		-\beta\Big(\frac{\rho_{\max}}{\rho}-1\Big)
		\bigg)
		\left[
		\frac{\beta\rho_{\max}
			\bigl(\beta\rho_{\max}-\gamma\rho\bigr)}
		{\rho^2}
		-\gamma
		\right].
	\end{align*}
The inequality 
	$	
	\mu(\rho)C(\rho)<0
	$ 
	holds if and only if	
	$    
	\gamma\rho^2
	+
	\beta\gamma\rho_{\max}\rho
	-
	\beta^2\rho_{\max}^2
	>0.$ 
	Equivalently, we have the relation 
	$
	\mu(\rho)C(\rho)<0
	\ \text{for}\
	\rho>\rho_c,
	$
	where the critical density reads as
	\[
	\rho_c
	=
	\frac{\beta\rho_{\max}}{2}
	\left(
	\sqrt{1+\frac{4}{\gamma}}-1
	\right).
	\]
	Hence, provided that the inequality \(\rho_c<\rho_{\max}\) holds, the
	Newell-Franklin fundamental diagram satisfies the desired
	Hamilton-Jacobi sign condition for all densities
	$
	\rho\in(\rho_c,\rho_{\max}],
	$
	apart from the degenerate point
	\(\rho=\beta\rho_{\max}/\gamma\), where \(R'(\rho)=0\) and the
	representation of \(C(\rho)\) in terms of \(R''/R'\) is singular. 
	However, we have the undesirable property~$\mu(\rho)>0$ for light traffic, which motivates  the following example.
\end{example}

\medskip

\medskip

\begin{example}[PIU closure] 
	Finally, we state an equilibrium velocity that is specifically tailored to the PIU model. More precisely, Corollary~\ref{CorollaryCpos} states positivity of $C(\rho)$ only for small densities and not for all, which is the case in the previous examples. This weaker statement allows to construct velocities that can introduce artificial instabilities to model stop-and-go waves in heavy traffic. In particular, the PIU velocity presented here is inspired by a modified  Newell-Franklin model. 	
	We assume that the density~$\rho\in[0,1]$ is normalized 
	and we consider the equilibrium velocity
	\begin{equation}\label{ConditionsMATTEO}	
	\begin{aligned}
	&	V(\rho)
		=
		\omega_0
		\left[
		1-
		\rho^\beta
		\exp\Big(
		\alpha\frac{\rho-1}{\rho+\lambda}
		\Big)
		\right]
		\\ 
	&	\text{with}
		\quad
		\lambda> 0,
		\quad
		0<\beta<\gamma,
		\quad
		\frac{\alpha}{1+\lambda}
		>
		\gamma-\beta
	\end{aligned}		
	\end{equation}
	and positive constants $\alpha\coloneqq \nicefrac{|w|}{\omega_0}>0$, where 
	the parameter $ w<0$  denotes the propagation speed of the backward-traveling  congestion wave. It holds
	$$
	\textup{sign}\Big\{\mu(\rho)\Big\}
	=
	\textup{sign}
	\left\{
	\beta
	+
	\frac{\alpha(1+\lambda)\rho}{(\rho+\lambda)^2}
	-
	\gamma
	\right\}. 
	$$ 
	Under the conditions~\eqref{ConditionsMATTEO} there exists exactly one solution $\rho^*\in (0,\lambda)$ where the term~$\mu(\rho)$  changes sign. 
	In the special case $\lambda\geq 1$ it holds~$
	C(\rho)>0$ for all large densities 
	$\rho\in(\rho^*,1).
	$

	Since the definition of $C(\rho)$ contains the divisor~$a'\big(\HAR(\rho)\big)$, 
	which vanishes for~$\mu(\rho^*)=0$,
	the function~$C(\rho)$ contains a pole at~$\rho^*$. 
	The singularity of $C(\rho)$ at $\rho^*$ is only due to
	the representation involving the quotient $R''(\rho)/R'(\rho)$ and
	does not induce a singularity in the Hamilton-Jacobi coefficient
	$\mu(\rho)C(\rho)$. Indeed, we have
	\begin{align*}
		&\mu(\rho)=\rho P(\rho)R'(\rho),
		\quad
		C(\rho)
		=
		2+\rho\frac{R''(\rho)}{R'(\rho)}
		+\rho\frac{P'(\rho)}{P(\rho)}, \\
		&\Rightarrow\quad
		\mu(\rho)C(\rho)
		=
		2\rho P(\rho)R'(\rho)
		+\rho^2P(\rho)R''(\rho)
		+\rho^2P'(\rho)R'(\rho).
	\end{align*}
	Hence, the product $\mu(\rho)C(\rho)$ admits a finite continuation
	through $\rho^*$, namely
	\begin{equation*}    
		\bigl(\mu C\bigr)(\rho^*)
		=
		(\rho^*)^2P(\rho^*)R''(\rho^*)
	\end{equation*}
	provided that~$R''(\rho)$ is continuous. 
	For the present closure this quantity is strictly positive.
	Thus,  the critical value $\rho^*$ is not a singular point of the Hamilton-Jacobi
	correction $\big(\mu C\big)(\rho)$ itself. 
	Table~\ref{TableExampleI} illustrates these findings for the case~$\lambda\geq 1$ and 
	Table~\ref{TableExampleII} for~$\lambda< 1$, respectively.

	\begin{table}[h]
		\centering
		\caption{Sign structure of the velocity-gradient diffusion coefficient
			$\mu(\rho)$ and correction $C(\rho)$ for the parameters
			$\beta=0.80$, $\gamma=2.15$, $\lambda=1.00$,
			$\alpha=3.0375$. The critical density at which
			$\mu(\rho)$ changes sign is
			$\rho^*=0.50$,
			while the zero of $C(\rho)$ below $\rho^*$ is
			$\rho_C^-\approx0.35$.}
		\begin{tabular}{lcp{5mm}p{5mm}cc}
			\hline
			Traffic regime & Density range & $\mu(\rho)$ & $C(\rho)$
			& $\mu(\rho)C(\rho)$ & acceleration \\
			\hline
			
			Light traffic
			& $0<\rho<\rho_C^-$
			& $-$ & $+$ & $-$
			& $ \lambda_1^{\textup{HJ}}
			\geq
			\lambda_1^{\textup{eq}}$
			\\[2pt]
			
			Intermediate regime
			& $\rho_C^-<\rho<\rho^*$
			& $-$ & $-$ & $+$
			& $ \lambda_1^{\textup{HJ}}
			\leq
			\lambda_1^{\textup{eq}}$
			\\[2pt]
			
			Stop-and-go 
			& $\rho^*<\rho<1$
			& $+$ & $+$ & $+$
			& $ \lambda_1^{\textup{HJ}}
			\leq
			\lambda_1^{\textup{eq}}$
			\\
			\hline
		\end{tabular}
		
		\vspace{2mm}
		\begin{minipage}{0.9\textwidth}
			\footnotesize
			The critical density $\rho^*$ is a pole of  $C(\rho)$, which is caused by
			the property $R'(\rho^*)=0$. There is no second zero of the function $C(\rho)$ in the
			heavy traffic density interval $(\rho^*,1)$. 
			The last column states the ordering of characteristic speeds
			under the acceleration condition
			$\partial_x v^{\mathrm{eq}}>0$.
		\end{minipage}
		\label{TableExampleI}
	\end{table}
	
	\begin{table}[h]
		\centering
		\caption{Sign structure of the velocity-gradient diffusion coefficient
			$\mu(\rho)$ and correction $C(\rho)$ for the parameters
			$\beta=0.50$, $\gamma=1.50$, $\lambda=0.50$,
			$\alpha=1.515$. The critical density at which
			$\mu(\rho)$ changes sign is
			$\rho^*\approx0.243$.
			The function $C(\rho)$ has two zeros,
			$\rho_C^-\approx0.152$ and
			$\rho_C^+\approx0.762$, with
			$\rho_C^-<\rho^*<\rho_C^+$.}
		
		\begin{tabular}{lcp{5mm}p{5mm}cc}
			\hline
			Traffic regime & Density range & $\mu(\rho)$ & $C(\rho)$
			& $\mu(\rho)C(\rho)$ & acceleration \\
			\hline
			
			Light traffic
			& $0<\rho<\rho_C^-$
			& $-$ & $+$ & $-$
			& $\lambda_1^{\textup{HJ}}
			\geq
			\lambda_1^{\textup{eq}}$
			\\[2pt]
			
			Intermediate regime 
			& $\rho_C^-<\rho<\rho^*$
			& $-$ & $-$ & $+$
			& $\lambda_1^{\textup{HJ}}
			\leq
			\lambda_1^{\textup{eq}}$
			\\[2pt]
			
			\multirow{2}{*}{Stop-and-go}
			& $\rho^*<\rho<\rho_C^+$
			& $+$ & $-$ & $-$
			& $\lambda_1^{\textup{HJ}}
			\geq
			\lambda_1^{\textup{eq}}$
			\\[2pt]
			
			& $\rho_C^+<\rho<1$
			& $+$ & $-$ & $-$
			& $\lambda_1^{\textup{HJ}}
			\geq
			\lambda_1^{\textup{eq}}$
			\\
			\hline
		\end{tabular}
		
		\vspace{2mm}
		\begin{minipage}{0.9\textwidth}
			\footnotesize
			The critical density $\rho^*$ is a pole of $C(\rho)$, which is caused by
			the property $R'(\rho^*)=0$. In contrast to the case
			$\lambda\geq1$ in Table~\ref{TableExampleI}, the present choice $\lambda=0.50<1$ yields a second zero
			of $C(\rho)$ in the heavy-traffic interval, namely
			$\rho_C^+\approx0.762>\rho^*\approx0.243$.
			The last column states the ordering of characteristic speeds
			under the acceleration condition
			$\partial_x v^{\mathrm{eq}}>0$.
		\end{minipage}
		\label{TableExampleII}
	\end{table}
	
	The sign structure of $\mu(\rho)$ and $C(\rho)$ plays an important role in the emergence and propagation of traffic waves. For regimes with low densities when it holds~$\mu(\rho)<0$  the diffusive relaxation is dissipative and damps small-scale perturbations, hence   suppressing the formation of stop-and-go waves. As the density increases, the parameter $\mu(\rho)$ changes sign at the critical density $\rho^*$ and becomes positive in the congested regime. The stabilizing effect of diffusion is then lost, allowing perturbations to grow and giving rise to stop-and-go instabilities. In addition, the sign change of $\mu(\rho)C(\rho)$ modifies the Hamilton-Jacobi correction to the characteristic wave speeds. Thus, the transition from $\mu C<0$ to $\mu C>0$ marks a qualitative change in the propagation of traffic disturbances and is consistent with the distinct wave behavior observed in the low- and high-density regimes. 
	To illustrate the effect of the Hamilton-Jacobi correction, we consider
	first the light-traffic regime, where the inequality $\mu(\rho)C(\rho)<0$ holds. The identity 
	$
	\lambda_1^{\textup{HJ}}-\lambda_1^{\textup{eq}}
	=
	-\varepsilon\mu(\rho)C(\rho)\partial_x \veq,
	$ yields the implications
	$$
	\partial_x \veq>0
	\quad\Rightarrow\quad
	\lambda_1^{\textup{HJ}}
	\geq
	\lambda_1^{\textup{eq}}
	\qquad\quad
	\text{and}
	\qquad\quad
	\partial_x \veq<0
	\quad\Rightarrow\quad
	\lambda_1^{\textup{HJ}}
	\leq
	\lambda_1^{\textup{eq}}.
	$$
	Thus, in light traffic, finite relaxation shifts the first
	characteristic speed towards larger values during acceleration and
	towards smaller values during braking. 
	In the heavy-traffic regime,  the ordering 
	might be reversed. 
	Hence, this closure produces a density-dependent reversal of the
	finite-relaxation correction to the propagation of traffic
	disturbances.

\end{example}

\begin{remark}[Kinetic traffic models]
	The density-dependent change of stability exhibited by the previous
	example is consistent with mechanisms arising in kinetic traffic
	models. In particular, in \cite{Herty:2020}, the
	equilibrium speed is obtained as the first moment of a kinetic
	equilibrium distribution and exhibits a transition between a free-flow
	and a congested regime. The corresponding Chapman-Enskog analysis
	identifies density intervals in which the effective diffusion changes
	its stability properties, providing a mechanism for the formation of
	bounded stop-and-go waves.
	
	The kinetic equilibrium velocity is not globally smooth at the
	transition density and therefore does not directly satisfy the
	regularity assumptions of Theorem~\ref{TheoremAwRascle2}. Nevertheless, its qualitative
	behaviour motivates the construction above, namely,  the sign change of
	the relaxation correction with increasing density. The PIU closure therefore provides, within the present smooth macroscopic framework, an analogue of the density-dependent stability transition generated by the kinetic model.
\end{remark}

Summarizing, the three functions~$V(\rho)$, $P(\rho)$, $R(\rho)$ must be specified. Due to the Relation $R(\rho) P(\rho) = \omega_0-V(\rho)$ it is sufficient to specify two of them. 
These examples state appropriate candidate functions that can be calibrated by a fundamental diagram to real-world data, which yields a calibration to a first-order LWR model. In particular, the equilibrium velocity $V(\rho)$ and equilibrium acceleration parameter~$R(\rho)$ are physically meaningful quantities, while the traffic pressure is implicitly characterized by them. 

The term~$\HAR(\rho)$ describes the equilibrium driver response at a given traffic density, representing the average acceleration or reaction behavior of drivers as a dimensionless quantity.
The function $a(h)$, in turn, translates this dimensionless driver-related state into the corresponding effective velocity-related quantity, which  enters the traffic pressure term. 
In equilibrium only the composition $a\big(
\HAR(\rho)
\big)$ and not the individual functions are of interest. Furthermore, the third equations in the PIU model~\eqref{PiuS1} also satisfies the material derivative~$\partial_t a(h) +  v\partial_x a(h) =0$ in terms of all choices~$a()$. Therefore, it makes sense to use the  simplest choice, namely~$
a(h)=
\alpha_0 h,
$ 
where the constant~$\alpha_0$ denotes a reference speed, while the driver-related state variable  $h$ is  dimensionless.

\section{Illustration of the theoretical results}
\label{SectionNumerics}
In this section we show the solution to a Riemann problem with constant left~$\u_\ell$
and right~$\u_r$
state as  explained in Theorem~\ref{TheoremRP}. 
The left panel of Figure~\ref{Figure1} shows the solution to the Riemann problem with initial values
$
\rho_\ell
=0.8
$, 
$
\rho_r
=0.1
$, 
$
v_\ell
=0.2
$, 
$
v_r
=0.5
$,~$h_\ell=\rho_\ell^{\nicefrac{3}{2}}$,~$h_r=\rho_r^{\nicefrac{3}{2}}$ and initial jump at~$x_0=0$. 
It states the density~$\rho$ (blue), velocity~$v$ (red) and the driver-dependent quantity~$h$ (green) with respect to the left $y$-axis. 
The first characteristic speed~$\lambda_1(\u)$ is plotted with respect to the scale of the right $y$-axis. 
Since it is  increasing, characteristics do not cross and a rarefaction wave occurs and not a shock.

If the Rankine-Hugoniot condition~$
f(\u_\ell) - f(\bar{\u})
=s(\u_\ell - \bar{\u})
$ holds, the left state is connected to the intermediate state by a shock wave with speed~$s>0$. Note that in this case the intermediate states~$\bar{v}$ and~$\bar{h}$  can be obtained as in the previous discussion for the rarefaction wave, namely by making use of the Riemann invariants and the linear degeneracy of the third field. 
The right panel of Figure~\ref{Figure1} shows the solution with 
initial values 
$
\rho_\ell
=0.3
$, 
$
\rho_r
=0.7
$, 
$
v_\ell
=0.5
$, 
$
v_r
=0.2
$ and~$h_\ell=\rho_\ell^{\nicefrac{3}{2}}$, $h_r=\rho_r^{\nicefrac{3}{2}}$. 
Here, characteristics corresponding to the first field with eigenvalues~$\lambda_1(\u)$ cross and a shock occurs, which is followed by a contact corresponding to the second and third field.

\begin{figure}[H]

	\begin{minipage}{0.49\textwidth}
		
		\begin{center}	
			\textbf{(i) \ Rarefaction wave} 
			\vspace{-1mm}			
			\scalebox{1}{\includegraphics[width=\linewidth]{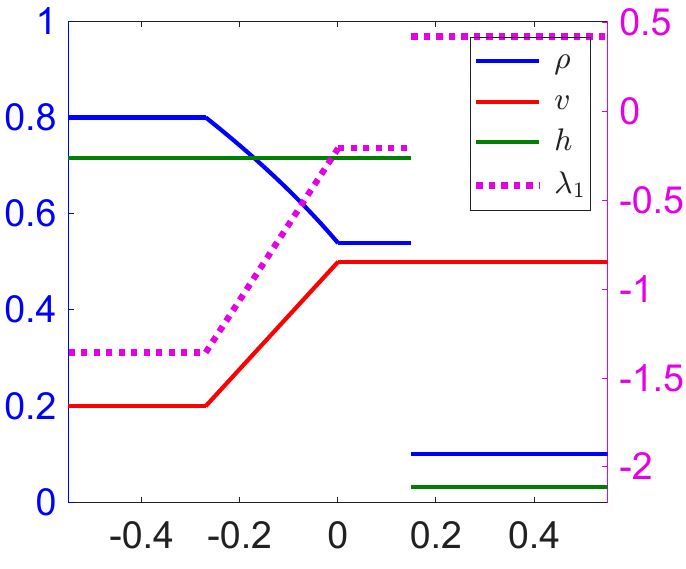}}

		\end{center}	
	\end{minipage}
	\hfil
	\begin{minipage}{0.49\textwidth}
		
		\begin{center}
			\textbf{(ii) \ Shock wave} 
			\vspace{-1mm}			
			\scalebox{1}{\includegraphics[width=\linewidth]{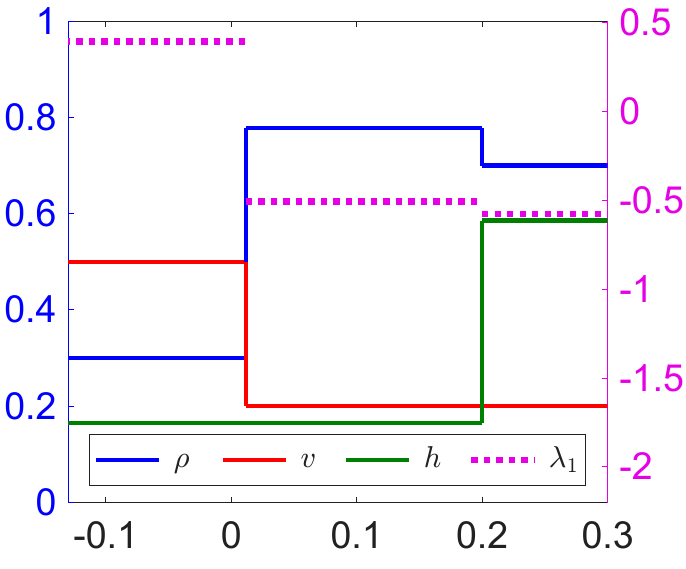}}	
			
		\end{center}
	\end{minipage}
	
	\caption{Solution to the Riemann problem of the PIU model~\eqref{PIUconservative} for~$\gamma = 2$ and a(h)=1+h. The left panel (i) states the solution to the Riemann problem at time $t=0.3$. Panel (ii) states the shock case at time $t=1$.}
	
	\label{Figure1}	
\end{figure}

\section*{Conclusion} 
The proposed framework provides a methodology to incorporate driver-dependent effects directly into a macroscopic traffic flow model. The analysis shows that introducing an additional dynamic variable leads to a richer characteristic and stability structure than in standard second-order models. In particular, the second-order Chapman-Enskog expansion reveals higher-order effects that are not visible at the leading diffusive level. The resulting stability conditions highlight the role of the underlying constitutive relations and provide a basis for further investigating the influence of driver heterogeneity in traffic dynamics.

\newpage

\bibliographystyle{abbrv}
\bibliography{trafficbib}

@article{Albeaik:2022,
author = {S.~Albeaik and A.~Bayen and M.~T.~Chiri and X.~Gong and A.~Hayat and N.~Kardous and A.~Keimer and S.~T.~McQuade and B.~Piccoli and Y.~You},
title = {Limitations and Improvements of the {I}ntelligent {D}river {M}odel ({IDM})},
journal = {SIAM Journal on Applied Dynamical Systems},
volume = {21},
number = {3},
pages = {1862-1892},
year = {2022},
doi = {10.1137/21M1406477}
}

@Article{Gong:2023,
title = {On the well-posedness of the ``{B}ando-follow the leader'' car following model and a time-delayed version},
journal = {Networks and Heterogeneous Media},
volume = {18},
number = {2},
pages = {775-798},
year = {2023},
issn = {1556-1801},
doi = {10.3934/nhm.2023033},
author = {X.~Gong and A.~Keimer}
}

@article{Piu:2022,
title = {Stability analysis of microscopic models for traffic flow with lane changing},
journal = {Networks and Heterogeneous Media},
volume = {17},
number = {4},
pages = {495-518},
year = {2022},
issn = {1556-1801},
doi = {10.3934/nhm.2022006},
author = {M.~Piu and G.~Puppo}
}

@ARTICLE{Hayat:2023,
	author = {A.~Hayat and B.~Piccoli and S.~Truong},
	title = {Dissipation of traffic jams using a single autonomous vehicle on a ring road},
	year = {2023},
	journal = {SIAM Journal on Applied Mathematics},
	volume = {83},
	number = {3},
	pages = {909-937},
	doi = {10.1137/21M1414346}
}

@ARTICLE{Gazis:1961,
	author = {D.~C.~Gazis and R.~Herman and R.~W.~Rothery},
	title = {Nonlinear Follow-the-Leader Models of Traffic Flow},
	year = {1961},
	journal = {Operations Research},
	volume = {9},
	number = {4},
	pages = {545--567},
	doi = {10.1287/opre.9.4.545}
}

@article{Bagnerini:2003,
author = {P.~Bagnerini and M.~Rascle},
title = {A Multiclass Homogenized Hyperbolic Model of Traffic Flow},
journal = {SIAM Journal on Mathematical Analysis},
volume = {35},
number = {4},
pages = {949-973},
year = {2003},
doi = {10.1137/S0036141002411490}
}

@article{Corli:2019,
author = {A.~Corli and H.~Fan},
title = {Hysteresis and stop-and-go waves in traffic flows},
journal = {Mathematical Models and Methods in Applied Sciences},
volume = {29},
number = {14},
pages = {2637-2678},
year = {2019},
doi = {10.1142/S0218202519500568}
}

@article{Corli:2023,
author = {A.~Corli and H.~Fan},
title = {The hysteretic {A}w-{R}ascle-{Z}hang model},
journal = {Studies in Applied Mathematics},
volume = {153},
number = {4},
pages = {e12769},
doi = {10.1111/sapm.12769},
year = {2024}
}

@Article{Lighthill:1955,
  Title                    = {On kinematic waves. {II}. {A} theory of traffic flow on long crowded roads},
  Author                   = {M.~J.~Lighthill and G.~B.~Whitham},
  Journal                  = {Proceedings of the Royal Society of London. Series A},
  Year                     = {1955},
  Pages                    = {317-345},
  Volume                   = {229},
doi = {10.1098/rspa.1955.0089}
}

@Article{Richards:1956,
  Title                    = {Shock waves on the highway},
  Author                   = {P.~I.~Richards},
  Journal                  = {Operations Research},
  Year                     = {1956},
  Pages                    = {42-51},
  Volume                   = {4}
}

@Article{Aw:2000,
  Title                    = {Resurrection of ``second order'' models of traffic flow},
  Author                   = {A.~Aw and M.~Rascle},
  Journal                  = {SIAM Journal on Applied Mathematics},
  Year                     = {2000},
  Number                   = {3},
  Pages                    = {916-938},
  Volume                   = {60},
  ISSN                     = {0036-1399},
doi = {10.1137/S0036139997332099}
}

@Article{Zhang:2002,
	Title                    = {A non-equilibrium traffic model devoid of gas-like behavior},
	Author                   = {H.~M.~Zhang},
	Journal                  = {Transportation Research Part B: Methodological},
	Year                     = {2002},
	Number                   = {3},
	Pages                    = {275-290},
	Volume                   = {36},
    doi                      = {10.1016/S0191-2615(00)00050-3}
}

@ARTICLE{Cardaliaguet:2021,
	author = {P.~Cardaliaguet and N.~Forcadel},
	title = {From heterogeneous microscopic traffic flow models to macroscopic models},
	year = {2021},
	journal = {SIAM Journal on Mathematical Analysis},
	volume = {53},
	number = {1},
	pages = {309-322},
	doi = {10.1137/20M1314410}
}

@ARTICLE{Cristiani:2016,
	author = {E.~Cristiani and S.~Sahu},
	title = {On the micro-to-macro limit for first-order traffic flow models on networks},
	year = {2016},
	journal = {Networks and Heterogeneous Media},
	volume = {11},
	number = {3},
	pages = {395-413},
	doi = {10.3934/nhm.2016002}
}

@article{Burger:2018,
title = {Derivation of a first order traffic flow model of {L}ighthill-{W}hitham-{R}ichards type},
journal = {IFAC-PapersOnLine},
volume = {51},
number = {9},
pages = {49-54},
year = {2018},
doi = {10.1016/j.ifacol.2018.07.009},
author = {M.~Burger and S.~G\"{o}ttlich and T.~Jung}
}

@article{Aw:2002,
author = {A.~Aw and A.~Klar and M.~Rascle and T.~Materne},
title = {Derivation of Continuum Traffic Flow Models from Microscopic Follow-the-Leader Models},
journal = {SIAM Journal on Applied Mathematics},
volume = {63},
number = {1},
pages = {259-278},
year = {2002},
doi = {10.1137/S0036139900380955}
}

@article{Zanella2019,
  author  = {A.~Tosin and M.~Zanella},
  title   = {Kinetic-Controlled Hydrodynamics for Traffic Models with Driver-Assist Vehicles},
  journal = {Multiscale Modeling \& Simulation},
  volume  = {17},
  number  = {2},
  pages   = {716--749},
  year    = {2019},
  doi     = {10.1137/18M1203766}
}

@article{Marques:2013,
title = {On the kinetic theory of vehicular traffic flow: {C}hapman-{E}nskog expansion versus {G}rad's moment method},
journal = {Physica A},
volume = {392},
number = {16},
pages = {3430-3440},
year = {2013},
issn = {0378-4371},
doi = {10.1016/j.physa.2013.03.052},
author = {W.~Marques and A.~R.~M\'{e}ndez}
}

@article{Iannini:2016,
    author = {M.~L.~L.~Iannini and R.~Dickman},
    title = {Kinetic theory of vehicular traffic},
    journal = {American Journal of Physics},
    volume = {84},
    number = {2},
    pages = {135-145},
    year = {2016},
    issn = {0002-9505},
    doi = {10.1119/1.4935895}
}

@article{Lu:2025,
title = {Modeling dynamics of traffic flow, information creation and spread through vehicle-to-vehicle communications: {A} kinetic approach},
journal = {International Journal of Non-Linear Mechanics},
volume = {175},
pages = {105096},
year = {2025},
issn = {0020--7462},
doi = {10.1016/j.ijnonlinmec.2025.105096},
author = {S.~Lu}
}

@article{ChenLevermoreLiu1994,
  author  = {G.-Q.~Chen and C.D.~Levermore and T.-P.~Liu},
  title   = {Hyperbolic Conservation Laws with Stiff Relaxation Terms and Entropy},
  journal = {Communications on Pure and Applied Mathematics},
  volume  = {47},
  number  = {6},
  pages   = {787--830},
  year    = {1994},
  doi     = {10.1002/cpa.3160470602}
}

@article{Borsche:2018,
author = {R.~Borsche and A.~Klar},
title = {A Nonlinear Discrete Velocity Relaxation Model For Traffic Flow},
journal = {SIAM Journal on Applied Mathematics},
volume = {78},
number = {5},
pages = {2891-2917},
year = {2018},
doi = {10.1137/17M1152681}
}

@article{Herty:2020,
title = {The {BGK} approximation of kinetic models for traffic},
journal = {Kinetic and Related Models},
volume = {13},
number = {2},
pages = {279-307},
year = {2020},
doi = {10.3934/krm.2020010},
author = {M.~Herty and G.~Puppo and S.~Roncoroni and G.~Visconti}
}

@BOOK{BRESSAN,
	author    = {A.~Bressan},
	title     = {Hyperbolic systems of conservation laws: The one dimensional {C}auchy problem},
	year      = {2005},
	isbn      = {0198507003},
	publisher = {Oxford University Press},
	address   = {New York},
	Series    = {Oxford {L}ecture {S}eries in {M}athematics and its {A}pplications}
}

@article{Crandall1983,
	author 	= {M.~G.~Crandall and P.-L.~Lions},
	journal = {Transactions of the American Mathematical Society},
	pages 	= {1-42},
	publisher = {American Mathematical Society},
	title 	= {Viscosity Solutions of {H}amilton-{J}acobi Equations},
	volume 	= {277},
	year 	= {1983}
}

@article{Ostia26,
	author 	= {Stephan Gerster and Giuseppe Visconti},
	journal = {Communications in Mathematical Sciences},
	pages 	= {1031-1051},
	publisher = {International Press},
	title 	= {Relaxation and stability analysis of a third-order multiclass traffic flow model},
	volume 	= {24},
	number  = {4},
	year 	= {2026}
}

@article{Dimarco2021,
	author 	= {Giacomo Dimarco and Andrea Tosin and Mattia Zanella},
	journal = {Journal of Statistical Physics},
	pages 	= {1572-9613},
	title 	= {Kinetic derivation of {A}w-{R}ascle-{Z}hang-type traffic models with driver-assist vehicles},
	volume 	= {186},
	number  = {17},
	year 	= {2021}
}

@article{Dimarco2019,
	author 	= {Giacomo Dimarco and Andrea Tosin},
	journal = {Journal of Statistical Physics},
	pages 	= {178-210},
	title 	= {The {A}w-{R}ascle traffic Model: Enskog-type kinetic derivation and generalisations},
    year    = {2019},
    volume  = {178}
}

@article{Goatin2026,
	author 	= {Paola Goatin and Stefan Moreti},
	journal = {Nonlinear Differential Equations and Applications NoDEA},
	pages 	= {1-53},
	title 	= {Well-posedness and {G}odunov approximation for nonlinear conservation laws with hysteresis},
	volume 	= {33},
	number  = {131},
	year 	= {2026}
}

@article{GoatinRizzo2024,
	author    = {Paola Goatin  and Alessandra Rizzo},
	title     = {Instabilities in generic second order traffic models with relaxation},
	journal   = {Zeitschrift f{\"u}r Angewandte Mathematik und Physik},
	volume    = {75},
	number    = {5},
	pages     = {167},
	year      = {2024}
}

@article{Goatin2019,
	author    = {Paola Goatin  and Nicolas Laurent-Brouty},
	title     = {The zero relaxation limit for the Aw-Rascle-Zhang traffic flow model},
	journal   = {Zeitschrift f{\"u}r Angewandte Mathematik und Physik},
	volume    = {70},
	number    = {31},
	pages     = {1-24},
	year      = {2019}
}

@article{Emiliano2021,
	author    = {Maya Briani and Emiliano Cristiani and Paolo Ranut},
	title     = {Macroscopic and multi-scale models for multi-class vehicular dynamics with uneven space occupancy: A case study},
	journal   = {Axioms},
	volume    = {10},
	number    = {2},
	pages     = {1-21},
	year      = {2021}
}

@article{Emiliano2019,
	author 	= {Emiliano Cristiani  and Maria Cristina Saladino},
	title 	= {Comparing comparisons between vehicular traffic states in microscopic and macroscopic first-order models},
	journal = {Mathematical Methods in the Applied Sciences},
	volume 	= {42},
	number 	= {3},
	pages 	= {918-934},
	year 	= {2019}
}

@article{Cristiani2025,
	title 	= {Kinetic description and macroscopic limit of swarming dynamics with continuous leader-follower transitions},
	journal = {Mathematics and Computers in Simulation},
	volume 	= {228},
	pages 	= {362-385},
	year 	= {2025},
	author 	= {Emiliano Cristiani and Nadia Loy and Marta Menci and Andrea Tosin}
}

@article{Coclite2005Traffic,
	author    = {G.~M.~Coclite and M.~Garavello and B.~Piccoli},
	title     = {Traffic Flow on a Road Network},
	journal   = {SIAM Journal on Mathematical Analysis},
	volume    = {36},
	number    = {6},
	pages     = {1862-1886},
	year      = {2005},
	doi       = {10.1137/S0036141004402683}
}

@article{CocliteRisebro2007,
	author  = {G.~M.~Coclite and N.~H.~Risebro},
	title   = {Viscosity Solutions of {H}amilton-{J}acobi Equations
	with Discontinuous Coefficients},
	journal = {Journal of Hyperbolic Differential Equations},
	volume  = {4},
	number  = {4},
	pages   = {771-795},
	year    = {2007}
}

@inproceedings{CocliteKarlsen2012,
	author    = {G.~M.~Coclite and K.~H.~Karlsen},
	title     = {{H}amiltonian Approximation of Entropy Solutions of the {B}urgers Equation},
	booktitle = {Hyperbolic Problems: Theory, Numerics and Applications. Volume 1},
	series    = {Series in Contemporary Applied Mathematics (CAM)},
	volume    = {17},
	pages     = {160--171},
	publisher = {World Scientific Publishing},
	address   = {Singapore},
	year      = {2012}
}

@article{Coclite2010VanishingVF,
	title={Vanishing Viscosity for Traffic on Networks},
	author={G.~M.~Coclite and M.~Garavello},
	journal={SIAM Journal on Mathematical Analysis},
	year={2010},
	volume={42},
	pages={1761-1783},
	url={https://api.semanticscholar.org/CorpusID:45302870}
}

@article{Coclite2021a,
	title = {Vanishing viscosity for a $2\times 2$ system modeling congested vehicular traffic},
author = {G.~M.~Coclite and N.~D.~Nitti and M.~Garavello and F.~Marcellini},
	journal = {Networks and Heterogeneous Media},
	volume = {16},
	number = {3},
	pages = {413-426},
	year = {2021}
}

\end{document}